\documentclass[10pt,a4paper,reqno,oneside]{amsart}
\usepackage{tikz}
\usepackage{tikz-cd}
\usepackage{comment}
\usepackage{hyperref}
\usepackage{cleveref}
\usepackage{ amssymb }
\usepackage[top=2.5cm]{geometry}
\usetikzlibrary{decorations.markings}

\usepackage{caption}

\theoremstyle{plain}
\newtheorem{thm}{Theorem}[section]

\newtheorem{lem}[thm]{Lemma}
\newtheorem{cor}[thm]{Corollary}

\newtheorem{prop}[thm]{Proposition}

\theoremstyle{definition}

\newtheorem{rem}[thm]{Remark}

\newcommand{\R}{\mathbb{R}}
\newcommand{\C}{\mathbb{C}}

\newcommand{\N}{\mathbb{N}}

\newcommand{\HF}{\mathrm{HF}}
\newcommand{\CF}{\mathrm{CF}}
\newcommand{\Ham}{\mathrm{Ham}}

\title{Bounding the Lagrangian Hofer metric via barcodes}
\author{Patricia Dietzsch}
\address{Department of Mathematics\\
  ETH Zürich\\
  Rämistrasse 101, 8092 Zürich, Switzerland}
\email{patricia.dietzsch@math.ethz.ch}
\thanks{The author was partially supported by the Swiss National
Science Foundation (grant number  200021\_204107)}
\begin{document}
\maketitle

\tikzset{middlearrow/.style={
        decoration={markings,
            mark= at position 0.5 with {\arrow{#1}} ,
        },
        postaction={decorate}
    }
}

%%%%%%%%%%%%%%%%%%%% Abstract %%%%%%%%%%%%%%%%%%%%%%%%
\begin{abstract}
We provide an upper bound on the Lagrangian Hofer distance between equators in the cylinder in terms of the barcode of persistent Floer homology. The bound consists of a weighted sum of the lengths of the finite bars and the spectral distance.
\end{abstract}

%%%%%%%%%%%%%%%%%%%% Introduction %%%%%%%%%%%%%%%%%%%%
\section{Introduction and main results}
Let $(M,\omega=-\mathrm{d}\lambda)$ be an exact symplectic manifold.
Consider the group $\Ham (M)$ of compactly supported Hamiltonian diffeomorphisms.
Any compactly supported Hamiltonian function $H\in C^{\infty}([0,1] \times M)$
generates a Hamiltonian flow $\{\phi_t^H\}$.
\begin{comment}
defined by
\[
    \frac{\mathrm{d}}{\mathrm{d}t}\phi_t^H(x) = X_t^H(\phi_t^H(x)),
    \qquad \omega(X_t^H, -) = -\mathrm{d}H_t.
\]
\end{comment}
The Hofer norm of a Hamiltonian diffeomorphism $\phi \in \Ham(M)$
is given by
\[
    \vert \vert \phi \vert \vert _H = \inf
    \left \{ \int_0^1 \max_{x\in M} H_t(x) - \min_{x\in M} H_t(x) \, \mathrm{d}t
    \, \big \vert \, \phi_1^H = \phi  \right\}.
\]
Let $L$ and $L'$ be closed connected Lagrangian submanifolds in $M$ that are Hamiltonian isotopic. 
The Lagrangian Hofer distance between $L$ and $L'$ is defined by 
%infimizing the Hofer norm over all Hamiltonian diffeomorphisms that take $L$ to $L'$:
\[
    d_H(L,L') = \inf \left\{ \vert \vert \phi \vert \vert_H \,\big \vert \, \phi(L)=L' \right\} .
\]

For transversely intersecting and exact Lagrangians $L$ and $L'$ we consider the Floer complex $\CF (L,L')$ over $\mathbb{Z}_2$.
A choice of primitives of the exact $1$-forms $\lambda \vert_{L}$ and $\lambda \vert_{L'}$
gives rise to an action functional $\mathcal{A}$ that induces a filtration on the Floer complex $\CF(L,L')$. Therefore, the homology group $\HF(L,L')$
of $\CF(L,L')$ becomes a persistence module
$\HF^{\leq \bullet}(L,L')$. The barcode $\mathcal{B}(L,L')$ associated with it gives rise to a number of invariants for the pair $(L,L')$.
One of them is the Lagrangian spectral metric $\gamma(L,L')$, which is the largest distance between two infinite bars in $\mathcal{B}(L,L')$
\cite{Viterbo, KislevShelukhin}. The boundary depth $\beta_1(L,L')$, studied
in \cite{usher}, is the length of the longest finite bar.
In this paper, we also consider the lengths of the other finite bars.
We denote by
\[
    \beta_1(L,L') \geq \beta_2(L,L') \geq \dots \geq \beta_{k}(L,L')
\]
the lengths of the finite bars ordered by their size.
%If $L$ and $L'$ do not intersect transversely, $\gamma(L,L')$ and $\beta_i(L,L')$
%can still be defined, by perturbing $L'$ by arbitrarily Hofer-small Hamiltonian diffeomorphisms $\phi$.
While the barcode $\mathcal{B}(L,L')$ actually depends on the choice of primitives
of $\lambda \vert_L$ and $\lambda \vert_{L'}$,
the numbers $\gamma(L,L')$ and $\beta_i(L,L')$, $i \in \{ 1, \dots, k \}$, are independent of it.
Kislev-Shelukhin \cite{KislevShelukhin} proved the following inequalities 
%\textit{Check out for which mfds!}
%For all weakly exact Lagr, in fact for all wide weakly monotone Lagr. In particular for exact Lagr
\begin{align}\label{eq:beta_vs_gamma}
    \beta_1(L,L') \leq \gamma(L,L') \leq d_H(L,L').
\end{align}

In this paper we prove a converse inequality for equators in the cylinder.
From now on, we work in $\Sigma:=S^1 \times (-1,1)$. In local coordinates $(q,p)$, the standard symplectic form on $\Sigma$ is $\omega = \mathrm{d}q \wedge \mathrm{d}p = -\mathrm{d}\lambda$ for
$\lambda=p\mathrm{d}q$. Let $L_0=S^1\times \{0\} \subset \Sigma$ denote the zero-section. We are interested into the set $\mathcal{L}(L_0)$ of all Lagrangians $L\subset \Sigma$ which are Hamiltonian isotopic to $L_0$.  

Our main result is:

\begin{thm}\label{thm:main}
Suppose that $L,L'\in \mathcal{L}(L_0)$ intersect transversely in $2n$ points. Then 
\[
    d_H(L,L') \leq \sum_{j=1}^{n-1} 2^{j}\beta_j(L,L') + \gamma(L,L').
\]
\end{thm}

Using the inequalities (\ref{eq:beta_vs_gamma}) we get the following bound:
\begin{cor}\label{cor:main}
For any $L,L'$ as above
\[
    \gamma(L,L') \leq d_H(L,L') \leq 2^{n}\gamma(L,L').
\]
\end{cor}

\subsection{Relation to previous work.}

Theorem \ref{thm:main} was inspired by the following result of Khanevsky.
\begin{thm}[\cite{KhanevskyThesis}\footnote{Khanevsky proved this result for a wider class of surfaces and Lagrangians,
not just the cylinder.}]\label{thm:Khanevsky}
There exist constants $k$ and $c$ such that for any transversely intersecting $L, L' \in \mathcal{L}(L_0)$,
\[
d_H(L,L') \leq k \cdot \#(L\cap L') + c.
\]
\end{thm}
%\begin{rem}
%    Khanevsky proved this result for a wider class of surfaces and Lagrangians.
%    He also improved the linear bound to a logarithmic bound for equators in the sphere.
%\end{rem}
\noindent
Moreover, Khanevsky proved that Hofer's distance on $\mathcal{L}(L_0)$ is unbounded
\cite{Khanevsky}. Let $L\in \mathcal{L}(L_0)$ and consider a sequence $\{L_n\} \subset \mathcal{L}(L_0)$ of Lagrangians, transverse to $L$, such that $d_H(L,L_n) \xrightarrow{n \to \infty} \infty$. It follows from Khanevsky's result that $\#(L \cap L_n) \xrightarrow{n \to \infty} \infty$.
In contrast to Hofer's metric $d_H$, the spectral distance $\gamma$ is bounded \cite{Shelukhin}.
Therefore, Corollary \ref{cor:main} shows that the number of bars in $\mathcal{B}(L,L_n)$
tends to $\infty$. This recovers $\#(L \cap L_n) \xrightarrow{n \to \infty} \infty$
because the intersection points are in bijection with the endpoints of the bars in $\mathcal{B}(L,L_n)$.

It is known that $\gamma$ is $C^0$-continuous \cite{BHS}, while $d_H$ is not.
Our result gives the following insight into convergence in Hofer's metric.
\begin{cor}\label{cor:C_0}
    Suppose $\{L_n\}_{n\in \N}$ is a sequence of Lagrangians, transverse to $L$, that $C^0$-converges to $L$. If the sequence $k_n:= \#( L \cap L_n )$ is bounded, then
    $\{L_n\}$ converges to $L$ in the Lagrangian Hofer metric.
\end{cor}
\subsection{Outline of proof}
\begin{comment}
Since both sides of the inequality are invariant under Hamiltonian isotopy, it is enough to prove the Theorem for $L$ and $L_0$. 
\end{comment}
%We explain the strategy to prove Theorem \ref{thm:main} assuming that $L$ and $L'$ intersect transversely. Suppose $L$ and $L'$ intersect in $2n$ points $(n \geq 1)$.
We explain the strategy to prove Theorem \ref{thm:main}.
Let $L,L'\in \mathcal{L}(L_0)$ be two transverse Lagrangians that intersect in
$2n$ points.
The action spectrum $\{\mathcal{A}(q) \vert q \in L \cap L'\}$ coincides
with the endpoints of the bars in $\mathcal{B}(L,L')$. Therefore, the barcode $\mathcal{B}(L,L')$ consists of $n-1$ finite bars and two infinite bars.
We prove the Theorem by induction on the number of intersection points.

\textit{Base case:} If there are only two intersection points, say $q$ and $p$, then Hofer's distance is equal to the area of one of the Floer strips connecting $q$ and $p$. This is also the difference between the action values of the two infinite bars in $\mathcal{B}(L,L')$, hence $d_H(L,L') = \gamma(L,L')$.

\textit{Induction Step:} Suppose that Theorem \ref{thm:main} holds for Lagrangians intersecting transversely in $2(n-1)$ points. Let $L$ and $L'$ be as above intersecting in $2n \geq 4$ points. We assume that any two intersection points $q\neq p\in L\cap L'$ satisfy $\mathcal{A}(q)\neq \mathcal{A}(p)$.
In order to use the induction hypothesis, we construct a Lagrangian $L''$ using Khanevsky's construction for \textit{deleting a leaf} \cite{Khanevsky}. A \textit{leaf} is a connected component of $\Sigma \backslash (L \cup L')$, which is bounded by one connected component of $L'\backslash L$ and one connected component of $L \backslash L'$. 
See Figure \ref{fig:leaf}  for an illustration and section \ref{sec:deletion} for more details. 
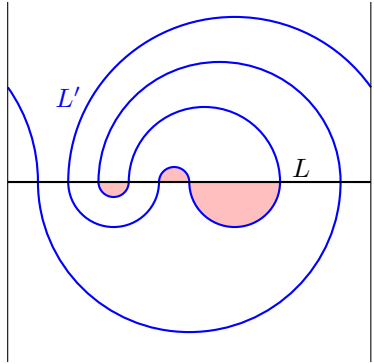
\begin{figure}[hbt] \label{fig:0}

\begin{tikzpicture}[thick,scale=0.398]

%Halbkurve -5 zu 5
%\fill[blue!10] (5,0) -- (360:5cm) arc (360:180:5cm) -- cycle;
\draw[blue] (5,0) -- (360:5cm) arc (360:180:5cm);

%Halbkurve 5 zu -3
%\fill[blue!10] (-5,0) -- (5,0) arc (0:180:4cm)-- cycle;
\draw[blue] (5,0) arc (0:180:4cm);

%Halbkurve -3 zu -2
\fill[red!25] (-3,0) -- (-2,0) arc (0:-180:0.5cm)-- cycle;
\draw[blue] (-2,0) arc (0:-180:0.5cm);

%Halbkurve -2 zu 3
%\fill[white] (-2,0) -- (3,0) arc (0:180:2.5cm)-- cycle;
\draw[blue] (3,0) arc (0:180:2.5cm);

%Halbkurve 3 zu 0
\fill[red!25] (0,0) -- (3,0) arc (0:-180:1.5cm)-- cycle;
\draw[blue] (3,0) arc (0:-180:1.5cm);

%Halbkurve 0 zu -1
\fill[red!25] (0,0) arc (0:180:0.5cm);
\draw[blue] (0,0) arc (0:180:0.5cm);

%Halbkurve -1 zu -4
\draw[blue] (-1,0) arc (0:-180:1.5cm);

%Halbkurve -5 zu -4
\draw[blue] (-5,0) arc (0:35:5.5cm);
\draw[blue] (-4,0) arc (180:35:5.5cm);

\draw (-6,0) -- (6,0);
\draw[thin] (-6,-6) -- (-6,6);
\draw[thin] (6,-6) -- (6,6);

\draw (-4,2) node[above=2pt, blue] {$L'$};
\draw (3.7,0) node[above=-2pt] {$L$};

\end{tikzpicture}
% \subcaption{Chart.}
\captionsetup{format=hang}
\caption{A pair of equators $(L,L')$ in the cylinder (the left and right vertical lines are identified). The three leaves are coloured in red.}
\label{fig:leaf}
\end{figure}

In section 4 we show the following

    \begin{prop}\label{lem:main}
    Let $[a,b)$ be the shortest finite bar in $\mathcal{B}(L,L')$. 
    Let $\bar{q},\bar{p} \in L\cap L'$ be the intersection points satisfying $\mathcal{A}(\bar{q})=b$
    and $\mathcal{A}(\bar{p})=a$. Then $\bar{q}$ and $\bar{p}$ are connected by a leaf.
   \end{prop}
    \noindent
    This leaf has area $\mathcal{A}(\bar{q})-\mathcal{A}(\bar{p}) = b-a = \beta_{n-1}(L,L')$.
    
We can therefore remove the intersection points $\bar{q}$ and $\bar{p}$ using Khanevsky's 
    construction for deleting a leaf: For any $\epsilon > 0$ there exists a Hamiltonian diffeomorphism $\phi$ of Hofer norm $\vert \vert \phi \vert \vert_H \leq \beta_{n-1}(L,L')+ \epsilon$ such that $L'':= \phi(L')$ intersects $L$ transversely and the number of intersection points is $2(n-1)$.
    
The Floer complexes $\CF(L,L')$ and $\CF(L,L'')$ can be endowed with action filtrations such that their persistent homologies are $\frac{\beta_{n-1}(L,L')+\epsilon}{2}$-interleaved. It follows that
    \begin{align} \label{ineq:beta}
         \vert \beta_j(L,L') - \beta_j(L,L'') \vert \leq \beta_{n-1}(L,L') + \epsilon
    \end{align}
    for all $1\leq j \leq n-2$ and
    \begin{align}\label{ineq:gamma}
        \vert \gamma(L,L') - \gamma(L,L'') \vert \leq \beta_{n-1}(L,L') + \epsilon.
    \end{align}

The theorem holds true for $(L,L'')$ by the induction hypothesis. 
    We therefore get
    \begin{align*}
    d_H(L,L') &\leq d_H(L,L'') + d_H(L'',L') \\
            &\leq \left( \sum_{j=1}^{n-2} 2^j \beta_j(L,L'') + \gamma(L,L'')\right) + \left(\beta_{n-1}(L,L') + \epsilon \right)\\
            &\leq \sum_{j=1}^{n-2} 2^j \left(\beta_j(L,L') + \beta_{n-1}(L,L') + \epsilon \right) \\
            & \qquad + \left(\gamma(L,L') + \beta_{n-1}(L,L') + \epsilon\right) +\beta_{n-1}(L,L') + \epsilon \\
            &= \sum_{j=1}^{n-2} 2^j \beta_j(L,L') + \gamma(L,L') + \left( \left(\sum_{j=1}^{n-2} 2^j\right) +2\right)(\beta_{n-1}(L,L')+\epsilon)\\
           % &= \sum_{j=1}^{n-2} 2^j \beta_j(L,L') + \gamma(L,L') + 2^{n-1}(\beta_{n-1}(L,L')+\epsilon)\\
            &= \sum_{j=1}^{n-1} 2^j \beta_j(L,L') +\gamma(L,L') + 2^{n-1}\epsilon         ,
    \end{align*}
where we used (\ref{ineq:beta}), (\ref{ineq:gamma}) in the third inequality. Taking the limit as $\epsilon \to 0$ finishes the induction step.

\begin{rem}
    The interleaving between $\HF(L,L')$ and $\HF(L,L'')$, as well as the inequalities (\ref{ineq:beta}) and (\ref{ineq:gamma}) are well-known and hold in wide generality under the name of \textit{stability},
    see for example \cite{KislevShelukhin, usher}.
    Instead of directly applying the general theory to our special case, we include in section \ref{sec:deletion} a combinatorial proof for equators in the cylinder.
\end{rem}
%\subsection{Further questions \textit{(To be removed)}.}
%$\input{questions.tex}
\subsection{Organisation of the paper}
In section \ref{sec:HF} we explain persistent Floer homology for Lagrangians in the cylinder,
using a combinatorial framework for Floer theory.
In section \ref{sec:deletion} we recall the process of deletion of a leaf and analyze its effect on persistent Floer homology and its barcodes. In particular,
we give a combinatorial proof for the inequalities (\ref{ineq:beta}) and (\ref{ineq:gamma}).
Proposition \ref{lem:main} is finally proved in section \ref{sec:smallest_bar}.

\subsection{Acknowledgements}
I'm deeply grateful to my advisor Paul Biran for suggesting this project and many helpful discussions. I would also like to thank Yusuke Kawamoto for his useful remarks on the paper and especially for pointing out Corollary \ref{cor:C_0} to me.
The author was partially supported by the Swiss National Science Foundation (grant number 200021 204107).

%%%%%%%%%%%%%%%%%%%% Section 1 - Persistence Floer Homology %%%%%%%%%%%%%%%%%%
\section{Persistent Floer homology}\label{sec:HF}
In this section we introduce the filtered Floer complex for transversely intersecting Lagrangians in $\mathcal{L}(L_0)$, its persistent homology and the invariants we extract from the barcode. We focus on our object of interest, the cylinder, even though the concepts make sense in much more generality.
We use combinatorial Floer homology for curves in surfaces as developed by
de Silva--Robbin--Salamon \cite{SilvaRobbinSalamon}.
\footnote{The setting in \cite{SilvaRobbinSalamon} does not always cover the case of two isotopic curves. However, the proofs for most of the statements we use in this paper carry over to our setting. Whenever not, we indicate a proof.}
\subsection{Lagrangian Floer complex}
 We follow closely
\cite{SilvaRobbinSalamon}.
Let $L,L' \in \mathcal{L}(L_0)$ such that $L$ intersects $L'$ transversely.
The Floer complex of the pair $(L,L')$ is a chain complex whose underlying $\mathbb{Z}_2$-vector space is generated by the intersection points.
More concretely, denoting $P=L\cap L'$ we define
\[
    \CF(L,L') = \bigoplus_{p\in P} \mathbb{Z}_2 p.
\]
The differential is obtained from counting the number of so-called smooth lunes connecting two intersection points.
We formalize this as follows.
Let 
\[
\mathbb{D} := \{z \in \C \, \vert \, \mathrm{Im}z \geq 0, \vert z \vert \leq 1\}
\]
be the standard half disc. Let $q,p \in P$. A \textit{smooth lune} from $q$ to $p$ is a smooth orientation-preserving immersion $u \colon \mathbb{D} \longrightarrow \Sigma$
satisfying the boundary conditions
\[
  u(\mathbb{D} \cap \R) \subseteq L,
  \qquad u(\mathbb{D}\cap S^1) \subseteq L',
  \qquad u(-1) = q,
  \qquad u(1) = p.
\]
Figure \ref{fig:lune} below shows an example of a smooth lune from $q$ to $p$.
\begin{figure}[hbt] \label{fig:0}

\begin{minipage}[t]{.5\linewidth}
    \begin{center}

\begin{tikzpicture}[thick,scale=0.4]

%%%%%%% Half Disc
\filldraw[fill=blue!10,draw=red!0] (-6,0) -- (6,0) -- (5.90884,1.041889) -- (5.6382,2.0521)-- 
(5.19615,3)-- (4.5963,3.8567) -- (4.24264,4.24264)
-- (3.8567,4.5963) -- (3,5.19615) -- (2.05212,5.63816)
-- (1.041889,5.90885) -- (0,6) --(-1.041889,5.90885) 
--(-2.05212,5.63816) -- (-3,5.19615)-- (-3.8567,4.5963)
-- (-4.24264,4.24264) --(-4.5963,3.8567) --
(-5.19615,3) -- (-5.6382,2.0521) -- (-5.90884,1.041889) -- (-6,0);

\draw[blue, domain=0:6] plot ({\x},{sqrt(36-\x^2)});
\draw[blue, domain=0:6] plot (-{\x},{sqrt(36-\x^2)});
\draw (-6,0) -- (6,0);

\fill[red] (-6,0) circle [radius=4pt];
\fill[green] (6,0) circle [radius=4pt];
\draw[red] (-6,0) node[above=0pt,left=1pt] {$-1$};
\draw[green] (6,0) node[above=0pt, right=1pt] {$1$};

\draw (0,0) node[below=1pt] {$\mathbb{D} \cap \R$};
\draw (0,6) node[above=1pt, blue] {$\mathbb{D} \cap S^1$};

\draw[green] (6,-4) node[above=0pt, right=1pt] {$ $};
\draw[->] (8,2.5) -- node[scale=1,above=0pt]{$u$} (10.5,2.5);

\end{tikzpicture}
\end{center}
  \end{minipage}%
  \begin{minipage}[t] {.5\linewidth}
   \begin{center}
\begin{tikzpicture}[thick,scale=0.4]

\filldraw[fill=blue!10,draw=red!0] 
%erste Halbkurve
(-5,0) -- (-4.924,-0.86824)
-- (-4.6985,-1.7101) -- (-4.3301,-2.5) --(-3.83022, -3.21394) --
(-3.2139,-3.83022) --(-2.5, -4.330127) -- (-1.71010, -4.698463)
-- (0,-5)  -- 
(1.71010, -4.698463) --(2.5, -4.330127) --
(3.2139,-3.83022) --(3.83022, -3.21394) -- (4.3301,-2.5)
-- (4.6985,-1.7101) -- (4.924,-0.86824)
--(5,0)
%zweite Halbkurve
-- (4.93923,0.69459) -- (4.75877, 1.36808)
--(4.4641,2) -- (3.5711,3.0642) -- (3, 3.4641)
--(2.36808, 3.75877) -- (1.69459, 3.939231)
-- (1,4) --(0.3054, 3.939231) -- (-0.3054, 3.939231) -- (-1, 3.4641) -- (-1.57115, 3.064178) --(-2.0641,2.57115) -- (-2.4641,2) -- (-2.75877, 1.36808)
--(-2.9392, 0.69459)
-- (-3,0)
%dritte Halbkurve
(-2,0) --(-1.962, 0.43412) -- (-1.849232, 0.85505)
-- (-1.6651, 1.25) -- (-1.41511, 1.606969)
--(-1.106979,1.91511)--(-0.75, 2.1650635)--
(-0.35505, 2.349232)
-- (0.5,2.5) --(0.9341, 2.46202)
-- (1.355, 2.34923) -- (1.75, 2.165064)
--(2.106969,1.915111) -- (2.41511,1.606969)
--(2.665, 1.25) -- (2.8492316, 0.85505036)
--(2.9620194,0.43412044)
-- (3,0) 
%vierte Halbkurve
-- (2.97721162,-0.26047227) -- (2.9095389,-0.51303)
-- (2.79903811, -0.75) -- (2.649068, -0.964181)
-- (2.464181, -1.149067) -- (2.25, -1.299038)
--(2.01303, -1.409539) -- (1.76047, -1.477211)
-- (1.5,-1.5)
-- (1.2395,-1.4772) -- (0.98697,-1.4095389)
-- (0.75,-1.299038) -- (0.53581859,-1.149067)
-- (0.350933,-0.964181) -- (0.200962,-0.75)
-- (0.090461,-0.51303021) -- (0.0227884,-0.260473)
-- (0,0) -- (-5,0);
%doppelte Fläche
\filldraw[fill=blue!25,draw=red!0] 
(-3,0) -- (-2,0) -- 
(-2.030154,-0.171) -- (-2.0669873,-0.25)
-- (-2.1169778,-0.3213938)
--(-2.1786062,-0.38302222) -- (-2.25,-0.4330127)
--(-2.32899,-0.46984) -- (-2.41318,-0.4924)
-- (-2.5,-0.5)
--(-2.58682,-0.4924039) -- (-2.671, -0.46985)
-- (-2.75,-0.4330127) -- (-2.82139,-0.383022)
-- (-2.883022,-0.32139) -- (-2.933,-0.25)
--(-2.9698,-0.171) -- (-2.9924,-0.0868)
--(-3,0);

\draw (-6,0) -- (6,0);
\draw[thin] (-6,-6) -- (-6,6);
\draw[thin] (6,-6) -- (6,6);

\draw[blue, domain=0:5] plot ({\x},{-sqrt(25-\x^2)});
\draw[blue, domain=0:5] plot (-{\x},{-sqrt(25-\x^2)});
\draw[blue, domain=1:5] plot ({\x},{sqrt(16-(\x-1)^2)});
\draw[blue, domain=1:5] plot ({2-\x},{sqrt(16-(\x-1)^2)});
\draw[blue, domain=0:0.5] plot ({\x - 2.5},{-sqrt(0.5^2- \x^2)});
\draw[blue, domain=0:0.5] plot ({-\x-2.5},{-sqrt(0.5^2-\x^2)});
\draw[blue, domain=0:2.5] plot ({\x + 0.5},{sqrt(2.5^2 - \x^2)});
\draw[blue, domain=0:2.5] plot ({-\x + 0.5},{sqrt(2.5^2 -\x^2)});
\draw[blue, domain=0:1.5] plot ({\x + 1.5},{-sqrt(1.5^2 - \x^2)});
\draw[blue, domain=0:1.5] plot ({-\x + 1.5},{-sqrt(1.5^2 -\x^2)});
\draw[blue, domain=0:0.5] plot ({\x - 0.5},{sqrt(0.5^2 - \x^2)});
\draw[blue, domain=0:0.5] plot ({-\x - 0.5},{sqrt(0.5^2 -\x^2)});
\draw[blue, domain=0:1.5] plot ({\x - 2.5},{-sqrt(1.5^2 - \x^2)});
\draw[blue, domain=0:1.5] plot ({-\x - 2.5},{-sqrt(1.5^2 -\x^2)});
\draw[blue, domain=0:4.5] plot ({\x + 1.5},{sqrt(5.5^2 - \x^2)});
\draw[blue, domain=0:5.5] plot ({-\x + 1.5},{sqrt(5.5^2 -\x^2)});
\draw[blue, domain=4.5:5.5] plot ({\x - 10.5},{sqrt(5.5^2 - \x^2)});

\fill[green] (-5,0) circle [radius=4pt];
\draw[green] (-5,0) node[below=5pt,left=1pt] {$p$};
\fill[red] (0,0) circle [radius=4pt];
\draw[red] (0,0) node[below=5pt,right=1pt] {$q$};

\draw (-4,2) node[above=2pt, blue] {$L'$};
\draw (1.5,0) node[above=0pt] {$L$};

\end{tikzpicture}
\end{center}
% \subcaption{Chart.}
  \end{minipage}
  \caption{A smooth lune from $q$ to $p$.}
  \label{fig:lune}
\end{figure}
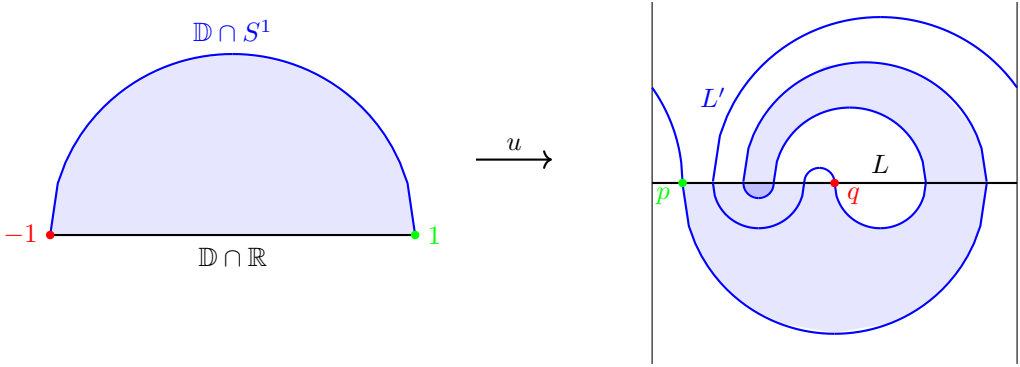

Two smooth lunes $u,u' \colon \mathbb{D} \longrightarrow \Sigma$ are called equivalent
if there exists an orientation-preserving diffeomorphism $\varphi \colon \mathbb{D} \longrightarrow \mathbb{D}$
such that $\varphi(1)=1$,
$\varphi(-1) = -1$
 and $u' = u \circ \varphi$.
We define $n(q,p) \in \mathbb{Z}_2$ to be the number mod $2$ of equivalence classes of smooth lunes from $q$ to $p$. The differential is then defined by
\[
   \partial (q) = \sum_{p\in P} n(q,p)p.
\]
It satisfies $\partial^2 =0$. 
\footnote{The proof in \cite{SilvaRobbinSalamon} carries over to this setting. It is based on studying \textit{broken hearts}, which are immersed discs with one non-convex corner. A quick proof can be found in \cite[Lemma 2.11]{Abouzaid}.} 
Therefore, $\left(\CF(L,L'), \partial \right)$ is a chain complex.

The following property of smooth lunes will be useful.
\begin{lem}\label{lem:lune}
    Let $u \colon \mathbb{D} \to \Sigma$ be a smooth lune from $q$ to $p$.
    Let $q=x_0, x_1, \dots, x_{l}, x_{l+1} = p$ be the points in $P \cap u(\mathbb{D}\cap S^1)$ ordered by their ordering on $L'$ when 
    following $L'$ from $q$ towards $p$ along $u(\mathbb{D} \cap S^1)$
    (see Figure \ref{fig:no_lune}). Assume $x_1, x_l \notin \{q,p\}$. Then $x_1$ and $x_{l}$ are not contained in 
    $u(\mathbb{D} \cap \R)$.
\end{lem}

\begin{proof}
    If not, $u$ is not an immersion at $u^{-1}(x_1)$ and $u^{-1}(x_l)$.
\end{proof}

\begin{figure}[hbt]\label{fig:0}

\begin{minipage}[t]{.5\linewidth}
    \begin{center}

\begin{tikzpicture}[thick,scale=0.48]

%Halbkurve -5 zu 5
\fill[blue!10] (5,0) -- (360:5cm) arc (360:180:5cm) -- cycle;
\draw[blue] (5,0) -- (360:5cm) arc (360:180:5cm);

%Halbkurve 5 zu -3
\fill[blue!10] (-5,0) -- (5,0) arc (0:180:4cm)-- cycle;
\draw[blue] (5,0) arc (0:180:4cm);

%Halbkurve -3 zu -2
\fill[blue!25] (-3,0) -- (-2,0) arc (0:-180:0.5cm)-- cycle;
\draw[blue] (-2,0) arc (0:-180:0.5cm);

%Halbkurve -2 zu 3
\fill[white] (-2,0) -- (3,0) arc (0:180:2.5cm)-- cycle;
\draw[blue] (3,0) arc (0:180:2.5cm);

%Halbkurve 3 zu 0
\fill[white] (0,0) -- (3,0) arc (0:-180:1.5cm)-- cycle;
\draw[blue] (3,0) arc (0:-180:1.5cm);

%Halbkurve 0 zu -1
\draw[blue] (0,0) arc (0:180:0.5cm);

%Halbkurve -1 zu -4
\draw[blue] (-1,0) arc (0:-180:1.5cm);

%Halbkurve -5 zu -4
\draw[blue] (-5,0) arc (0:35:5.5cm);
\draw[blue] (-4,0) arc (180:35:5.5cm);

\draw (-6,0) -- (6,0);
\draw[thin] (-6,-6) -- (-6,6);
\draw[thin] (6,-6) -- (6,6);

\fill[black] (-5,0) circle [radius=4pt];
\draw[black] (-5,0) node[below=5pt,left=1pt] {$p$};
\fill[black] (0,0) circle [radius=4pt];
\draw[black] (0,0) node[below=5pt,right=1pt] {$q$};

\fill[black] (-3,0) circle [radius=4pt];
\draw[black] (-3,0) node[above=5pt,left=-3pt] {$x_3$};
\fill[black] (-2,0) circle [radius=4pt];
\draw[black] (-2,0) node[above=5pt,right=-3pt] {$x_2$};

\fill[orange] (3,0) circle [radius=4pt];
\draw[orange] (3,0) node[above=5pt,right=-3pt] {$x_1$};
\fill[orange] (5,0) circle [radius=4pt];
\draw[orange] (5,0) node[above=5pt,right=-3pt] {$x_l$};

\draw (-4,2) node[above=2pt, blue] {$L'$};
\draw (1.5,0) node[above=0pt] {$L$};

\end{tikzpicture}
\end{center}
  \end{minipage}%
  \begin{minipage}[t] {.5\linewidth}
   \begin{center}
\begin{tikzpicture}[thick,scale=0.48]

%Halbkurve -5 zu 5
\fill[violet!30] (-4.5,0) -- (-1.5,0) arc (0:180:1.5cm) -- cycle;
\draw[blue] (-1.5,0) arc (0:180:1.5cm);

\fill[violet!30] (1.5,0) -- (4.5,0) arc (0:180:1.5cm) -- cycle;
\draw[blue] (4.5,0) arc (0:180:1.5cm);

\fill[violet!30] (-1.5,0) -- (1.5,0) arc (360:180:1.5cm) -- cycle;
\draw[blue] (1.5,0) arc (360:180:1.5cm);

%\fill[blue!10] (6,0) -- (4.5,0) arc (180:270:1.5cm) -- cycle;
\draw[blue] (4.5,0) arc (180:270:1.5cm);

%\fill[blue!10] (-6,0) -- (-4.5,0) arc (360:270:1.5cm) -- cycle;
\draw[blue] (-4.5,0) arc (360:270:1.5cm);

\fill[orange] (1.5,0) circle [radius=4pt];
\draw[orange] (1.5,0) node[above=6pt,left=-3pt] {$x_l$};

\fill[orange] (-1.5,0) circle [radius=4pt];
\draw[orange] (-1.5,0) node[above=6pt,right=-3pt] {$x_1$};

\fill[black] (-4.5,0) circle [radius=4pt];
\draw[black] (-4.5,0) node[below=6pt,left=1pt] {$q$};
\fill[black] (4.5,0) circle [radius=4pt];
\draw[black] (4.5,0) node[below=6pt,right=1pt] {$p$};

\draw (-6,0) -- (6,0);
\draw[thin] (-6,-6) -- (-6,6);
\draw[thin] (6,-6) -- (6,6);

\draw (-4,1) node[above=1pt, blue] {$L'$};
\draw (-5.5,0) node[above=0pt] {$L$};

\end{tikzpicture}
\end{center}
% \subcaption{Chart.}
  \end{minipage}
  \captionsetup{format=hang}
  \caption{A smooth lune on the left, where $x_1$ and $x_l$ do not lie in 
    $u(\mathbb{D}\cap \R)$. The violet area on the right can't be obtained by a smooth lune.}
    \label{fig:no_lune}
\end{figure}
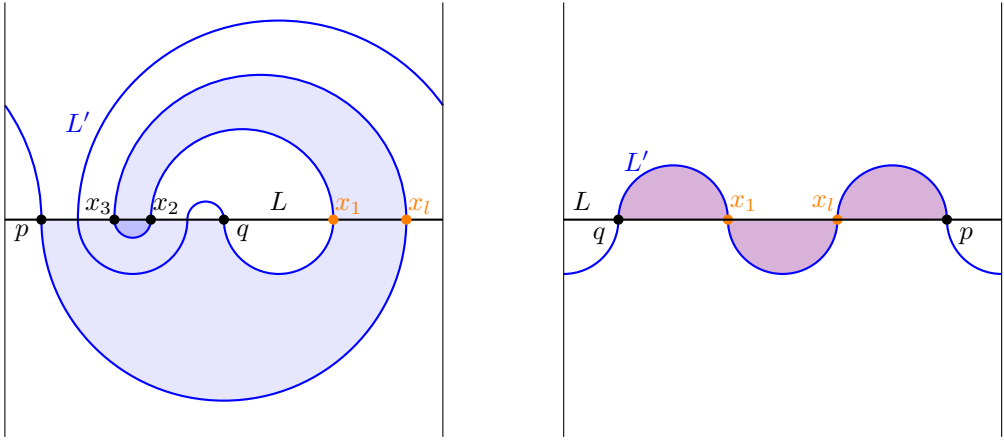

\begin{rem}
      \cite[Theorem 6.7]{SilvaRobbinSalamon} characterizes smooth lunes in terms of their boundary behavior. Here are some of its consequences.
      In a small enough neighbourhood $U$ of $q$ we may choose coordinates $(x,y)$ such that
      $L \cap U$ coincides with the $x$-axis and $L'\cap U$ coincides with
      the $y$-axis. If a smooth lune $u$ leaves $q$,  $u(\mathbb{D})\cap U$ lies entirely either in the first quadrant $(x\geq 0, y\geq 0)$ or in the
      third quadrant $(x \leq 0, y \leq 0)$.
      Similarly, if a lune $u$ enters $p$, then locally $u(\mathbb{D})$ lies
      either in the second quadrant $(x \leq 0, y \geq 0)$ or in the
      forth quadrant $(x \geq 0, y \leq 0)$.
      For any $q, p \in L \cap L'$, there are at most $2$ lunes
      from $q$ to $p$. 
      Figure \ref{fig:two_lunes} shows an example with $2$ lunes.
\end{rem}

\begin{figure}[hbt] \label{fig:0}

\begin{minipage}[t]{.5\linewidth}
    \begin{center}

\begin{tikzpicture}[thick,scale=0.48]

%Halbkurve -5 zu 5
\fill[blue!10] (5,0) -- (360:5cm) arc (360:180:5cm) -- cycle;
\draw[blue] (5,0) -- (360:5cm) arc (360:180:5cm);

%Halbkurve 5 zu -3
\fill[blue!10] (-5,0) -- (5,0) arc (0:180:4cm)-- cycle;
\draw[blue] (5,0) arc (0:180:4cm);

%Halbkurve -3 zu -2
\fill[blue!25] (-3,0) -- (-2,0) arc (0:-180:0.5cm)-- cycle;
\draw[blue] (-2,0) arc (0:-180:0.5cm);

%Halbkurve -2 zu 3
\fill[white] (-2,0) -- (3,0) arc (0:180:2.5cm)-- cycle;
\draw[blue] (3,0) arc (0:180:2.5cm);

%Halbkurve 3 zu 0
\fill[white] (0,0) -- (3,0) arc (0:-180:1.5cm)-- cycle;
\draw[blue] (3,0) arc (0:-180:1.5cm);

%Halbkurve 0 zu -1
\draw[blue] (0,0) arc (0:180:0.5cm);

%Halbkurve -1 zu -4
\draw[blue] (-1,0) arc (0:-180:1.5cm);

%Halbkurve -5 zu -4
\draw[blue] (-5,0) arc (0:35:5.5cm);
\draw[blue] (-4,0) arc (180:35:5.5cm);

\draw (-6,0) -- (6,0);
\draw[thin] (-6,-6) -- (-6,6);
\draw[thin] (6,-6) -- (6,6);

\fill[black] (-5,0) circle [radius=4pt];
\draw[black] (-5,0) node[below=5pt,left=1pt] {$p$};
\fill[black] (0,0) circle [radius=4pt];
\draw[black] (0,0) node[below=5pt,right=1pt] {$q$};

\draw (-4,2) node[above=2pt, blue] {$L'$};
\draw (1.5,0) node[above=0pt] {$L$};

\end{tikzpicture}
\end{center}
  \end{minipage}%
  \begin{minipage}[t] {.5\linewidth}
   \begin{center}
\begin{tikzpicture}[thick,scale=0.48]

%Halbkurve -5 zu -4
\fill[blue!10] (-6,0) -- (-5,0) arc (0:35:5.5cm)-- cycle;
\draw[blue] (-5,0) arc (0:35:5.5cm);
\fill[blue!10] (6,0) -- (-4,0) arc (180:35:5.5cm)-- cycle;
\draw[blue] (-4,0) arc (180:35:5.5cm);

%Halbkurve -1 zu -4
\fill[blue!10] (-4,0) -- (-1,0) arc (0:-180:1.5cm)-- cycle;
\draw[blue] (-1,0) arc (0:-180:1.5cm);

%Halbkurve -5 zu 5
\draw[blue] (5,0) -- (360:5cm) arc (360:180:5cm);

%Halbkurve 5 zu -3
\fill[blue!10] (-5,0) -- (5,0) arc (0:180:4cm)-- cycle;
\draw[blue] (5,0) arc (0:180:4cm);

%Halbkurve -3 zu -2
\fill[blue!10] (-3,0) -- (-2,0) arc (0:-180:0.5cm)-- cycle;
\draw[blue] (-2,0) arc (0:-180:0.5cm);

%Halbkurve -2 zu 3
\draw[blue] (3,0) arc (0:180:2.5cm);

%Halbkurve 3 zu 0
\fill[white] (0,0) -- (3,0) arc (0:-180:1.5cm)-- cycle;
\draw[blue] (3,0) arc (0:-180:1.5cm);

%Halbkurve 0 zu -1
\fill[white] (-1,0) -- (0,0) arc (0:180:0.5cm)-- cycle;
\draw[blue] (0,0) arc (0:180:0.5cm);

\draw (-6,0) -- (6,0);
\draw[thin] (-6,-6) -- (-6,6);
\draw[thin] (6,-6) -- (6,6);

\fill[black] (-5,0) circle [radius=4pt];
\draw[black] (-5,0) node[below=5pt,left=1pt] {$p$};
\fill[black] (0,0) circle [radius=4pt];
\draw[black] (0,0) node[below=5pt,right=1pt] {$q$};

\draw (-4,2) node[above=2pt, blue] {$L'$};
\draw (1.5,0) node[above=0pt] {$L$};

\end{tikzpicture}
\end{center}
% \subcaption{Chart.}
  \end{minipage}
  \caption{Two lunes from $q$ to $p$.}
    \label{fig:two_lunes}
\end{figure}
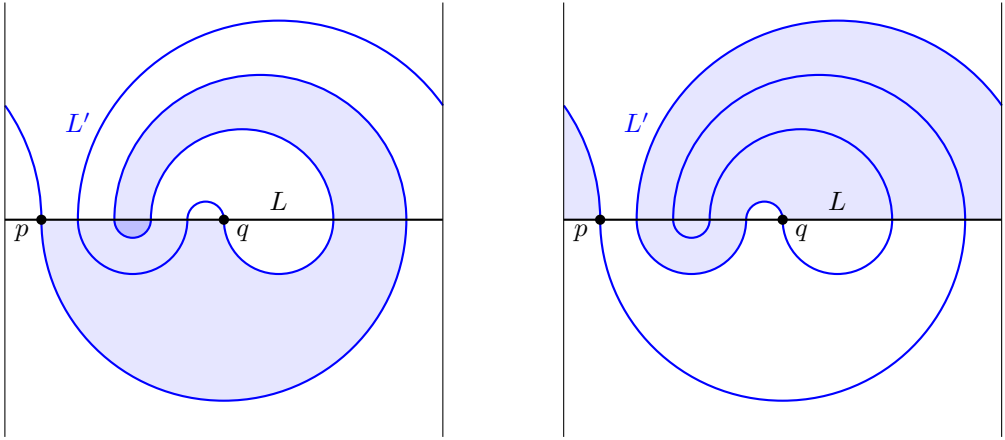
\subsection{Filtration}
Any $L\in \mathcal{L}(L_0)$ is exact, meaning that the $1$-form $\lambda\vert _{L} \in \Omega^1(L)$ is exact. We call a function 
\[
    h_L \colon L \longrightarrow \R
\]
such that $\mathrm{d}h_L = \lambda \vert _{L}$ a \textit{marking} of $L$. The marking $h_L$ is unique up to an additive constant because $L$ is connected.
The filtration on $\CF(L,L')$ we introduce below will depend on a choice
of markings of $L$ and $L'$.

Fix two markings $h_L$ and $h_{L'}$ of $L$ and $L'$ respectively.
Consider the space of paths from $L$ to $L'$, namely
\[
    \Omega_{L,L'} := \{ \gamma \in C^{\infty}([0,1]) \, \vert \, \gamma(0)\in L, \gamma(1) \in L'\}.
\]
We identify tangent vectors $\xi$ at $\gamma \in \Omega_{L,L'}$ with vector fields $\xi(t) \in T_{\gamma(t)}\Sigma$ along $\gamma$.
The action functional $\mathcal{A} \colon \Omega_{L,L'}\longrightarrow \R$ is defined by
\[
    \mathcal{A}(\gamma) = h_L(\gamma(1)) - h_{L'}(\gamma(0)) - \int_0^1 \lambda_{\gamma(t)}(\dot{\gamma}(t)) \, \mathrm{d}t.
\]
The exterior derivative of $\mathcal{A}$ is
\begin{align*}%\label{eq:action}
    \mathrm{d}\mathcal{A}_{\gamma}(\xi) = \int_0^1 \omega (\dot{\gamma}(t),\xi(t)) \, \mathrm{d}t
\end{align*}
for any $\xi \in T_\gamma\Omega_{L,L'}$.
We view $P$ as a subset of $\Omega_{L,L'}$ by viewing $p\in P$ as a constant path. Let $q,p\in P$ and let $u$ be a smooth lune from $q$ to $p$. It follows that
\begin{align}\label{eq:action_area}
    \int _{\mathbb{D}} u^*\omega = \mathcal{A}(q) - \mathcal{A}(p).
\end{align} 

\begin{rem}\label{rem:action}
    Any two neighbouring intersection points are connected by a smooth lune. More precisely, choose an orientation on $L$ and order the intersection 
    points $s_1,s_2, \dots, s_{2n}$ according to their order on $L$. Then for each $1 \leq i \leq 2n$, there exists a smooth lune from $s_i$ to $s_{i+1}$, or a smooth lune from 
   $s_{i+1}$ to $s_i$. (Here, we use cyclic notation for the indices, i.e.
   $s_{2n+1} = s_1$.) We explain these smooth lunes in section \ref{subsec:interleaving}.
   It follows that equation (\ref{eq:action_area}) determines the action functional on $P$ uniquely up to an additive constant.
\end{rem}

We have
$\int_{\mathbb{D}}u^*\omega \geq 0$ because $u$ is an orientation-preserving immersion. It follows that $$\mathcal{A}(q) - \mathcal{A}(p) \geq 0,$$ whenever there is a
lune from $q$ to $p$. In particular, the differential lowers filtration:
\[
    \mathcal{A}(\partial q) < \mathcal{A}(q).
\]
Therefore, for any $\alpha \in \R$, 
\[
    \CF^{\leq \alpha}((L,h_L),(L',h_{L'})) := \bigoplus_{ p \in P, \mathcal{A}(p) \leq \alpha} \mathbb{Z}_2 p
\]
is a subcomplex of $\CF(L,L')$.
Moreover, for each $\alpha \leq \beta$ there are inclusions
$$\CF^{\leq \alpha}((L,h_L),(L',h_{L'})) \subseteq \CF^{\leq \beta}((L,h_L),(L',h_{L'})).$$
The collection
$$\CF^{\leq \bullet }((L,h_L),(L',h_{L'}))
= \left \{ \CF^{\leq \alpha}((L,h_L),(L',h_{L'})) \right\}_{\alpha \in \R}
$$ is called the filtered Floer complex of the pair $\left((L,h_L),(L',h_{L'})\right)$.
\subsection{Persistent Floer homology and its barcode} \label{subsec:barcodes}
%The theory of persistence modules has been developed in topological data analysis \cite{carlsson} and successfully applied in symplectic topology.
For an overview of the theory of persistent homology and its use in symplectic topology see
\cite{PolterovichRosenSamvelyanZhang}. We follow closely chapters 1 and 2 from
\cite{PolterovichRosenSamvelyanZhang}.
\footnote{The definitions here differ slightly from those in \cite{PolterovichRosenSamvelyanZhang} in the convention for semicontinuity.
Since we consider $\CF^{\leq \alpha}$ and not $\CF^{<\alpha}$ we need to work
with intervals of the form $[a,b)$ and not $(a,b]$.}

A \textit{persistence module} over $\mathbb{Z}_2$ consists of an $\R$-indexed
family of $\mathbb{Z}_2$-vector spaces $\{V_t\}_{t\in \R}$ and 
linear maps $f_{s,t}\colon V_s \longrightarrow V_t$ for any $s \leq t$ satisfying
\begin{enumerate}
    \item $f_{t,t} = \mathrm{id}$ for any $t\in \R$,
    \item $f_{s,t}\circ f_{r,s} = f_{r,t}$ for any $r\leq s\leq t$.
\end{enumerate}
Any filtered chain complex gives rise to a 
persistence module by taking homology of the subcomplexes.
In our case, we consider the family 
\[
  \left \{\HF^{\leq \alpha}((L,h_L),(L',h_{L'}))\right\}_{\alpha \in \R} = \left \{\mathrm{H}_*\left(\CF^{\leq \alpha}((L,h_L),(L',h_{L'}))\right) \right \}_{\alpha \in \R}
\]
of $\mathbb{Z}_2$-vector spaces 
together with the maps
\[
    i_{\alpha, \beta}\colon \HF^{\leq \alpha}((L,h_L),(L',h_{L'})) \longrightarrow \HF^{\leq \beta}((L,h_L),(L',h_{L'}))
\]
induced by inclusion for any $\alpha \leq \beta$.
We refer to this persistence module as 
\[
  \HF^{\leq \bullet}((L,h_L),(L',h_{L'})).
\]

It is a persistence module of finite type, in the following sense.
A persistence module $V$ is \textit{of finite type} if 
\begin{enumerate}
    \item For any $t \in \R$, there exists $\epsilon >0$ such that 
    $f_{s,t}$ is an isomorphism for any
    $s\in [t,t+\epsilon)$.
    \item There exist $t_1, \dots, t_m \in \R$ such that for all $t\in \R \backslash \{t_1, \dots, t_m\}$, there exists $\epsilon >0$ such that
    $f_{s,r}$ is an isomorphism for any $s \leq r \in (t-\epsilon, t+\epsilon)$.
    \item There exists $s_-$ such that $V_s=0$ for any $s < s_-$.
\end{enumerate}
Isomorphism classes of persistence modules of finite type can be classified by their barcodes.
A \textit{barcode of finite type} $\mathcal{B} = \{ I_j\}_{j=1}^{n}$ is a finite multiset of intervals $I_j$ of two possible types:
\begin{enumerate}
    \item Finite bars: $I_j =[a_j,b_j)$, where $a_j < b_j$ are real numbers.
    \item Infinite bars: $I_j = [c_j, \infty)$, where $c_j$ is a real number.
\end{enumerate}
A bar $[a,b)$ corresponds to the persistence module $V^{[a,b)}$ defined by
\begin{align*}
    V^{[a,b)}_t = 
    \begin{cases}
           \mathbb{Z}_2 \qquad \text{ if } t\in [a,b),\\
           0 \qquad \quad \text{else},
    \end{cases}
\end{align*}
with internal maps $f_{s, t}\colon V_s^{[a,b)} \longrightarrow V_t^{[a,b)}$ being the identity for $a\leq s \leq t < b$ and $0$ for any other pair $s \leq t$.
\begin{thm}[Structure Theorem,\cite{ZomorodianCarlsson}]
    For a persistence module $V$ of finite type, there is an isomorphism
    of persistence modules
    \[ 
        V \cong \bigoplus_{I \in \mathcal{B}} V^I
    \]
    for a unique barcode $\mathcal{B} = \mathcal{B}(V)$ of finite type. 
\end{thm}
For a proof of this theorem, we recommend \cite[Chapter 2]{PolterovichRosenSamvelyanZhang}.
\footnote{There are more general versions for persistence modules not necessarily of finite type, see \cite{BotnanCrawley}.}

Applying the structure theorem to  $\HF^{\leq \bullet}((L,h_L),(L',h_{L'}))$ we get a barcode which we denote by
$\mathcal{B}((L,h_L),(L',h_{L'}))$.
The endpoints of the bars in $\mathcal{B}((L,h_L),(L',h_{L'}))$ are exactly the action values of the generators of $\CF(((L,h_L),(L',h_{L'}))$. Moreover, the infinite bars correspond to classes in $\HF(L,L') \cong \mathbb{Z}_2 \oplus \mathbb{Z}_2$.
If $L$ intersects $L'$ in $2n$ points, the barcode $\mathcal{B}((L,h_L),(L',h_{L'}))$ therefore consists of $n-1$ finite bars
and two infinite bars. 

In computations of the barcode, it is useful to use appropriate bases
for $\CF(L,L')$ as introduced and studied in \cite{Barannikov} and \cite{UsherZhang}.
A basis $e_1, \dots, e_{2n}$ of 
$\CF((L,h_L),(L',h_{L'}))$
is called \textit{orthogonal} if for all $\lambda_i \in \mathbb{Z}_2$
\[
    \mathcal{A}\left(\sum_{i=1}^{2n} \lambda_i e_i\right) = \max \left \{ \mathcal{A}(e_i) \, \vert \, \lambda_i \neq 0 \right \}.
\]
Following \cite{PolterovichRosenSamvelyanZhang} we call an orthogonal basis $e_1, \dots , e_{n-1}, f_1, \dots, f_{n-1}, g_1, g_2$ satisfying
\begin{align*}
        &\partial e_i = f_i \text{ for } 1\leq i \leq n-1 \\
        &\partial g_1 = \partial g_2 = 0
\end{align*}
a \textit{Jordan basis} for $\partial$.
The existence of such a basis is shown in \cite{Barannikov}.
\footnote{In \cite{UsherZhang} such a basis arises from a \textit{singular value decomposition} of $\partial \colon C \to \ker (\partial)$ for a Floer-type
chain complex $C$.}
The barcode can be read off the action values of an orthogonal
Jordan basis for $\partial$: $\mathcal{B}((L,h_L),(L',h_{L'}))$ consists of the finite bars $[\mathcal{A}(f_i), \mathcal{A}(e_i))$ and the infinite bars $[\mathcal{A}(g_i), \infty)$.

The barcode depends on $h_L$ and $h_{L'}$, but different choices of the markings yield the same barcode up to shift. Therefore, the length of each bar is independent of the markings. 
We denote by
\[
    \beta_1(L,L') \geq \beta_2(L,L') \geq \dots \geq \beta_{n-1}(L,L')
\]
the lengths of the finite bars. Similarly if $[c_1, \infty), [c_2,\infty)$ are the infinite bars, where $c_1\leq c_2$, then
\[
  \gamma(L,L') = c_2 - c_1
\]
is independent of the markings. The quantity $\gamma(L,L')$ is called the spectral
distance between $L$ and $L'$.

\section{Deletion of a leaf} \label{sec:deletion}
Throughout this section, let $L\in \mathcal{L}(L_0)$ be a Lagrangian that intersects $L_0$ transversely. 
For simplicity, we will only study the pair $(L_0,L)$ in the remainder of the paper. This is no restriction, because Hofer's distance and persistent 
Floer homology are invariant under Hamiltonian isotopies.
In this section we study the process of \textit{deletion of a leaf} introduced in \cite{Khanevsky}. 
\footnote{We only consider the zero-section $L_0$ in the cylinder $\Sigma$, while
Khanevsky's constructions work for more general curves and surfaces.}
After recalling this construction, we study its effect on persistent Floer
homology and its barcode.

\subsection{Hamiltonian diffeomorphism.}
Fix an orientation on $L_0$ and denote by $s_1, \dots, s_{2n}$ be the intersection points of $L_0 \cap L$, such that the ordering corresponds to their order on $L_0$.
We use cyclic notation for the indices (e.g. $s_{2n+1}=s_1$).
Write $[s_i, s_{j}]$ for the interval on $L_0$ with left end $s_i$ and right end $s_j$.
Khanevsky associates a graph $T(L)$ to $L$ consisting of two rooted trees, whose vertices carry weights and whose edges are oriented and ordered. Following closely \cite{Khanevsky} we recall this construction.

The graph $T(L)$ consists of 
\begin{itemize}
    \item one vertex for each connected component of $\Sigma \backslash (L_0 \cup L)$,
    \item each vertex $v$ carries a weight $a(v)\in (0, \infty]$ equal to the area of the corresponding region in $\Sigma$,
    \item an edge between $v_1$ and $v_2$, whenever the corresponding regions have a common boundary along a segment of $L_0\backslash L$,
    \item an orientation of the edge from the vertex corresponding to the region in the upper half of $\Sigma$ to the vertex corresponding to the region in the lower half of $\Sigma$,
    \item an ordering of the edges, by assigning the number $j\in \{1, \dots, 2n\}$ to the edge that corresponds to the segment $[s_j, s_{j+1}]$ on $L_0$.
\end{itemize}
In what follows, we often do not distinguish between a vertex and its corresponding component
in $\Sigma \backslash (L_0 \cup L)$. Denote the edge with number $j$ by $e_j$.
Figure \ref{fig:Trees} shows an example of $T(L)$.
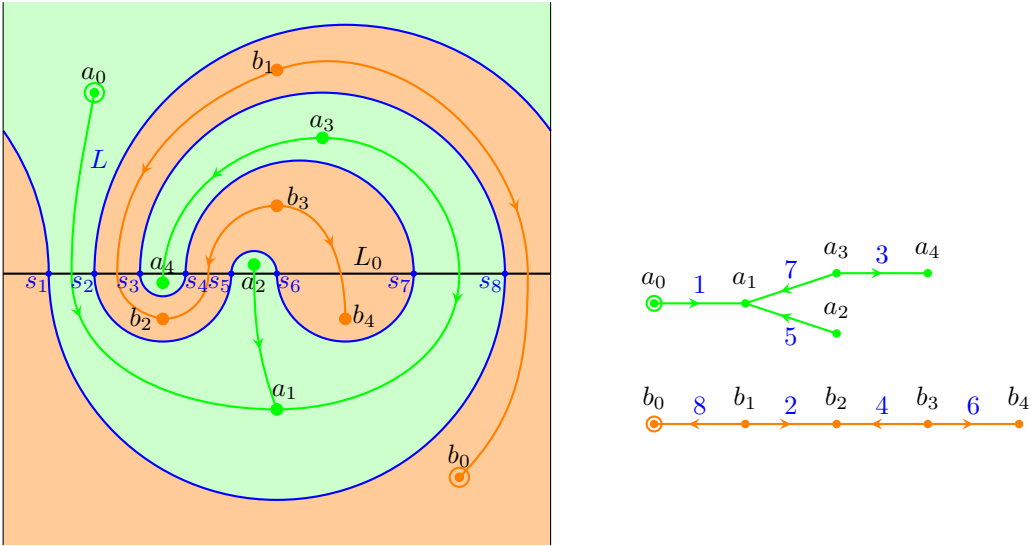
\begin{figure}[hbt] \label{fig:0}

\begin{minipage}[t]{.5\linewidth}
    \begin{center}
\begin{tikzpicture}[thick,scale=0.6]
%oben grün
\fill[green!20] (-6,6) -- (-6,0) -- (6,0) -- (6,6) --cycle;

%unten gelb
\fill[orange!40] (-6,-6) -- (-6,0) -- (6,0) -- (6,-6) -- cycle;

%Halbkurve -5 zu -4
\fill[orange!40] (-6,0) -- (-5,0) arc (0:35:5.5cm);
\fill[orange!40] (6,0) -- (-4,0) arc (180:35:5.5cm);
\draw[blue] (-5,0) arc (0:35:5.5cm);
\draw[blue] (-4,0) arc (180:35:5.5cm);

%Halbkurve -5 zu 5
\fill[green!20] (5,0) -- (360:5cm) arc (360:180:5cm) -- cycle;
\draw[blue] (5,0) -- (360:5cm) arc (360:180:5cm);

%Halbkurve 5 zu -3
\fill[green!20] (-5,0) -- (5,0) arc (0:180:4cm)-- cycle;
\draw[blue] (5,0) arc (0:180:4cm);

%Halbkurve -2 zu 3
\fill[orange!40] (-2,0) -- (3,0) arc (0:180:2.5cm)-- cycle;
\draw[blue] (3,0) arc (0:180:2.5cm);

%Halbkurve 3 zu 0
\fill[orange!40] (0,0) -- (3,0) arc (0:-180:1.5cm)-- cycle;
\draw[blue] (3,0) arc (0:-180:1.5cm);

%Halbkurve 0 zu -1
\fill[green!20] (-1,0) -- (0,0) arc (0:180:0.5cm);
\draw[blue] (0,0) arc (0:180:0.5cm);

%Halbkurve -1 zu -4
\fill[orange!40] (-4,0) -- (-1,0) arc (0:-180:1.5cm);
\draw[blue] (-1,0) arc (0:-180:1.5cm);

%Halbkurve -3 zu -2
\fill[green!20] (-3,0) -- (-2,0) arc (0:-180:0.5cm)-- cycle;
\draw[blue] (-2,0) arc (0:-180:0.5cm);

\draw (-6,0) -- (6,0);
\draw[thin] (-6,-6) -- (-6,6);
\draw[thin] (6,-6) -- (6,6);

\draw (-3.9,2) node[above=2pt, blue] {$L$};
\draw (2,0) node[above=-2pt] {$L_0$};

\fill[green] (-4,4) circle [radius=3pt];
\draw[green] (-4,4) circle [radius=6pt];
\fill[green] (0,-3) circle [radius=4pt];
\fill[green] (1,3) circle [radius=4pt];
\fill[green] (-2.5,-0.2) circle [radius=4pt];
\fill[green] (-0.5,0.2) circle [radius=4pt];

\draw[black] (-4,4) node[above=0pt] {$a_0$};
\draw[black] (0,-3) node[right=3pt,above=0pt] {$a_1$};
\draw[black] (1,3) node[above=-1pt] {$a_3$};
\draw[black] (-2.5,-0.2) node[above=0pt] {$a_4$};
\draw[black] (-0.5,0.2) node[below=1pt] {$a_2$};

\draw[green,middlearrow={stealth}] (-4,4) to [out=260,in=90] (-4.5,0) to [out=270,in=180] (0,-3);
\draw[green,middlearrow={stealth reversed}] (0,-3) to [out=0,in=270] (4,0) to [out=90,in=0] (1,3);
\draw[green,middlearrow={stealth}] (1,3)  to [out=180,in=90] (-2.5,-0.2);
\draw[green,middlearrow={stealth reversed}] (0,-3) to [out=110,in=270] (-0.5,0.2);

\fill[orange] (4,-4.5) circle [radius=3pt];
\draw[orange] (4,-4.5) circle [radius=6pt];
\fill[orange] (0,4.5) circle [radius=4pt];
\fill[orange] (-2.5,-1) circle [radius=4pt];
\fill[orange] (0,1.5) circle [radius=4pt];
\fill[orange] (1.5,-1) circle [radius=4pt];

\draw[orange,middlearrow={stealth reversed}] (4,-4.5) to [out=45,in=270] (5.5,0) to [out=90,in=20] (0,4.5);
\draw[orange,middlearrow={stealth}] (0,4.5) to [out=200,in=90] (-3.5,0) to [out=270,in=170] (-2.5,-1);
\draw[orange,middlearrow={stealth reversed}] (-2.5,-1) to [out=0,in=270] (-1.5,0)  to [out=90,in=180] (-0,1.5);
\draw[orange,middlearrow={stealth}] (0,1.5) to [out=0,in=90] (1.5,-1);

\draw[black] (4,-4.5) node[above=0pt] {$b_0$};
\draw[black] (0,4.5) node[left=5pt,above=-5pt] {$b_1$};
\draw[black] (-2.5,-1) node[left=0pt] {$b_2$};
\draw[black] (0,1.5) node[above=3pt,right=0pt] {$b_3$};
\draw[black] (1.5,-1) node[above=1pt,right=-1pt] {$b_4$};

\fill[blue] (-5,0) circle [radius=2pt];
\draw[blue] (-5,0) node[below=4pt,left=-4pt] {$s_1$};
\fill[blue] (-4,0) circle [radius=2pt];
\draw[blue] (-4,0) node[below=4pt,left=-4pt] {$s_2$};
\fill[blue] (-3,0) circle [radius=2pt];
\draw[blue] (-3,0) node[below=4pt,left=-4pt] {$s_3$};
\fill[blue] (-2,0) circle [radius=2pt];
\draw[blue] (-2,0) node[below=4pt,right=-4pt] {$s_4$};
\fill[blue] (-1,0) circle [radius=2pt];
\draw[blue] (-1,0) node[below=4pt,left=-4pt] {$s_5$};
\fill[blue] (0,0) circle [radius=2pt];
\draw[blue] (0,0) node[below=4pt,right=-3.5pt] {$s_6$};
\fill[blue] (3,0) circle [radius=2pt];
\draw[blue] (3,0) node[below=4pt,left=-3pt] {$s_7$};
\fill[blue] (5,0) circle [radius=2pt];
\draw[blue] (5,0) node[below=4pt,left=-3pt] {$s_8$};
%\draw[blue] (-2.5,-1) node[left=0pt] {$b_2$};
%\draw[blue] (0,1.5) node[above=3pt,right=0pt] {$b_3$};
%\draw[blue] (1.5,-1) node[above=1pt,right=-1pt] {$b_4$};

\end{tikzpicture}
    \end{center}
  \end{minipage}%
  \begin{minipage}[t] {.5\linewidth}
   \begin{center}
\begin{tikzpicture}[thick,scale=0.4]

\draw[green, middlearrow={stealth}] (0,2) -- (3,2);
\draw[green, middlearrow={stealth reversed}] (3,2) -- (6,3);
\draw[green, middlearrow={stealth}] (6,3) -- (9,3);
\draw[green, middlearrow={stealth reversed}] (3,2) -- (6,1);

\fill[green] (0,2) circle [radius=4pt];
\draw[green] (0,2) circle [radius=7pt];
\fill[green] (3,2) circle [radius=4pt];
\fill[green] (6,3) circle [radius=4pt];
\fill[green] (9,3) circle [radius=4pt];
\fill[green] (6,1) circle [radius=4pt];

\draw[black] (0,2) node[above=2pt] {$a_0$};
\draw[black] (3,2) node[above=2pt] {$a_1$};
\draw[black] (6,3) node[above=2pt] {$a_3$};
\draw[black] (9,3) node[above=2pt] {$a_4$};
\draw[black] (6,1) node[above=2pt] {$a_2$};

\draw[blue] (1.5,2) node[above=0pt] {$1$};
\draw[blue] (4.5,2.5) node[above=0pt] {$7$};
\draw[blue] (7.5,3) node[above=0pt] {$3$};
\draw[blue] (4.5,1.5) node[below=0pt] {$5$};

\draw[orange, middlearrow={stealth reversed}] (0,-2) -- (3,-2);
\draw[orange, middlearrow={stealth}] (3,-2) -- (6,-2);
\draw[orange, middlearrow={stealth reversed}] (6,-2) -- (9,-2);
\draw[orange, middlearrow={stealth}] (9,-2) -- (12,-2);

\fill[orange] (0,-2) circle [radius=4pt];
\draw[orange] (0,-2) circle [radius=7pt];
\fill[orange] (3,-2) circle [radius=4pt];
\fill[orange] (6,-2) circle [radius=4pt];
\fill[orange] (9,-2) circle [radius=4pt];
\fill[orange] (12,-2) circle [radius=4pt];

\draw[black] (0,-2) node[above=2pt] {$b_0$};
\draw[black] (3,-2) node[above=2pt] {$b_1$};
\draw[black] (6,-2) node[above=2pt] {$b_2$};
\draw[black] (9,-2) node[above=2pt] {$b_3$};
\draw[black] (12,-2) node[above=2pt] {$b_4$};

\draw[blue] (1.5,-2) node[above=0pt] {$8$};
\draw[blue] (4.5,-2) node[above=0pt] {$2$};
\draw[blue] (7.5,-2) node[above=0pt] {$4$};
\draw[blue] (10.5,-2) node[above=0pt] {$6$};

%phantom to move things up
\draw[yellow!0] (0,-6);

\end{tikzpicture}
% \subcaption{Chart.} 
    \end{center}
  \end{minipage}
    \caption{The graph $T(L)$ associated to a Lagrangian $L$.}
    \label{fig:Trees}
\end{figure}

The graph consists of two connected components corresponding to the two regions of $\Sigma \backslash L$. For each of these connected components, the vertex corresponding to the unbounded region is set to be the root. Thus each connected component is a rooted tree.
A leaf of $T(L)$, if different from the roots, corresponds to a region in
$\Sigma \backslash (L_0 \cup L)$ that is bounded by one connected component of $L_0\backslash L$ and one connected component of $L \backslash L_0$. It has corners at two neighboured intersection points. In the language of combinatorial Floer theory, this region corresponds to a minimal smooth lune $v\colon \mathbb{D} \longrightarrow \Sigma$
for $(L_0,L)$, in the sense that there is no smooth lune 
$u\colon \mathbb{D} \longrightarrow \Sigma$ with $\mathrm{Im}(u)
\subsetneq \mathrm{Im}(v)$.
Denote by $\bar{q}, \bar{p} \in L_0 \cap L$ the corners of the smooth lune $v$. We call $v$ a \textit{leaf} from $\bar{q}$ to $\bar{p}$ and identify it 
with the leaf in the tree.

\begin{rem}
    Lemma \ref{lem:main} states that a shortest bar gives rise to a leaf: If $[\mathcal{A}(\overline{p}),\mathcal{A}(\overline{q}))$
    is a shortest bar in $\mathcal{B}(L_0,L)$ then $\overline{q}$ and $\overline{p}$ are connected by a leaf.
    Conversely, if $\overline{q}$ and $\overline{p}$ are corners of a leaf \textit{with minimal area} among all leaves, then the interval 
    $[\mathcal{A}(\overline{p}),\mathcal{A}(\overline{q}))$
    is a bar in $\mathcal{B}(L_0,L)$.
    However, in general a (non-minimal) leaf does not correspond to a bar.    
\end{rem}

\begin{comment}
Vertices corresponding to regions adjacent to $\partial \Sigma$ are called stationary.
A leaf is called $contractible$ if it is not stationary and the corresponding region in $\Sigma \backslash (L_0 \cup L)$ is homeomorphic to a disc.

Khanevsky proves the following results:
\begin{prop} \label{prop:Kh}
Let $L$ be as above.
\begin{enumerate}
    \item $G(L)$ contains a contractible leaf. 
    \item Suppose $\Sigma \backslash L_0$ is disconnected. Then $G(L)$ has two connected components. Let $T$ be one of them. Let $v$ be a leaf with distance $2$ from a vertex $u$. Then there exists a Hamiltonian deformation of length at most $w(v) + \epsilon$ which removes $v$ from $T$, transfers the weight $w(v)$ to $u$. Some vertices might merge in this process, but the weights of the vertices different from $u$ and $v$ are changed by no more than $\epsilon$.
\end{enumerate}
\end{prop}

The way to do it is in the figure \ref{fig:Graphs}.
\begin{figure}[h]
    \centering
    \includegraphics[scale=0.5]{Pictures/Pic_Graphs.pdf}
    \caption{Graph.}
    \label{fig:Graphs}
\end{figure}

We denote by $e_j$ the edge labeled by $j$.
\end{comment}

The following result is due to Khanevsky:
\begin{prop}[\cite{Khanevsky}]\label{lem:khanevsky}
    Suppose there is leaf $v$ from $\bar{q}$ to $\bar{p}$ of area $a(v)$. 
    Let $w$ be a vertex at distance $2$ from $v$ in $T(L)$.
    Then for any $\epsilon>0$
    there exists a Hamiltonian diffeomorphism $\phi\in \Ham(\Sigma)$ such that $\vert \vert \phi \vert \vert
    _H < a(v) + \epsilon$ and which removes $v$ from the $T(L)$ by moving its weight $a(v)$ from $v$ to the vertex $w$. The support of $\phi$ is contained
    in a neighbourhood of $v$ and a curve as shown on the right in Figure \ref{fig:deletion}.
    In particular, $L_0 \cap \phi(L) = \left( L_0 \cap L \right) \backslash
    \{\bar{q}, \bar{p}\}$ and $\phi(x) = x $ for $x \in (L\cap L_0) \backslash
    \{\bar{q}, \bar{p}\}$.
\end{prop}
We refer to this Hamiltonian isotopy by \textit{deletion of a leaf}.
Figure \ref{fig:deletion} shows schematically how the deletion of a leaf looks like.
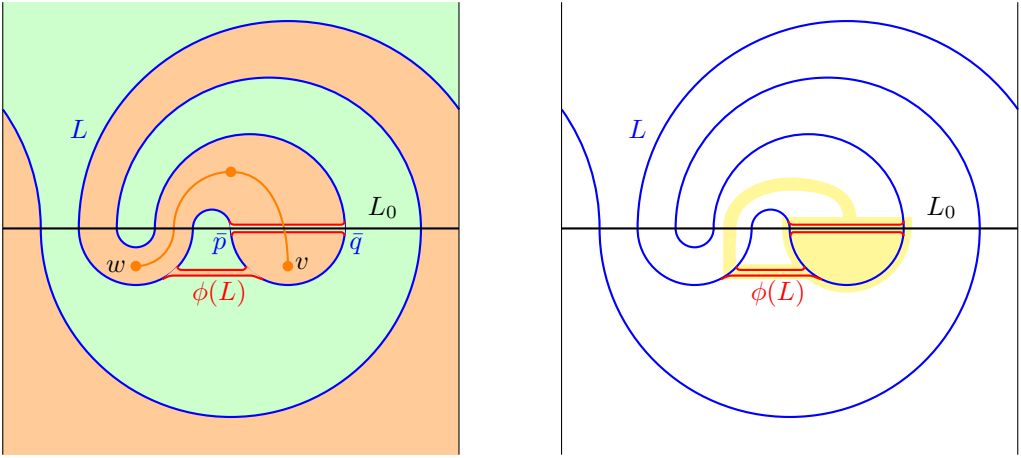
\begin{figure}[hbt] 
\begin{minipage}[t]{.5\linewidth}
   \begin{center}
  \begin{tikzpicture}[thick,scale=0.5]
%oben grün
\fill[green!20] (-6,6) -- (-6,0) -- (6,0) -- (6,6) --cycle;

%unten gelb
\fill[orange!40] (-6,-6) -- (-6,0) -- (6,0) -- (6,-6) -- cycle;

%Halbkurve -5 zu -4
\fill[orange!40] (-6,0) -- (-5,0) arc (0:35:5.5cm);
\fill[orange!40] (6,0) -- (-4,0) arc (180:35:5.5cm);
\draw[blue] (-5,0) arc (0:35:5.5cm);
\draw[blue] (-4,0) arc (180:35:5.5cm);

%Halbkurve -5 zu 5
\fill[green!20] (5,0) -- (360:5cm) arc (360:180:5cm) -- cycle;
\draw[blue] (5,0) -- (360:5cm) arc (360:180:5cm);

%Halbkurve 5 zu -3
\fill[green!20] (-5,0) -- (5,0) arc (0:180:4cm)-- cycle;
\draw[blue] (5,0) arc (0:180:4cm);

%Halbkurve -2 zu 3
\fill[orange!40] (-2,0) -- (3,0) arc (0:180:2.5cm)-- cycle;
\draw[blue] (3,0) arc (0:180:2.5cm);

%Halbkurve 3 zu 0
\fill[orange!40] (0,0) -- (3,0) arc (0:-180:1.5cm)-- cycle;
\draw[blue] (3,0) arc (0:-180:1.5cm);

%Halbkurve 0 zu -1
\fill[green!20] (-1,0) -- (0,0) arc (0:180:0.5cm);
\draw[blue] (0,0) arc (0:180:0.5cm);

%Halbkurve -1 zu -4
\fill[orange!40] (-4,0) -- (-1,0) arc (0:-180:1.5cm);
\draw[blue] (-1,0) arc (0:-180:1.5cm);

%Halbkurve -3 zu -2
\fill[green!20] (-3,0) -- (-2,0) arc (0:-180:0.5cm)-- cycle;
\draw[blue] (-2,0) arc (0:-180:0.5cm);

\draw[thin] (-6,-6) -- (-6,6);
\draw[thin] (6,-6) -- (6,6);

\draw (-4,2) node[above=2pt, blue] {$L$};
\draw (4,0) node[above=0pt] {$L_0$};
\draw (-0.3,-1.1) node[below=0pt, red] {$\phi(L)$};
\draw[blue] (0,0) node[below=6pt,left=-2pt] {$\bar{p}$};
\draw[blue] (3,0) node[below=6pt,right=-2pt] {$\bar{q}$};

\fill[orange!40] (-1.8,-1.35)  to [out=30,in=180] (-1.5,-1.25) to [out=0,in=180] 
(0.5,-1.25) to [out=0,in=170] (0.85,-1.35) -- (0.4,-1)
to [out=-90,in=0] (0.3,-1.1) to [out=0,in=180] (-1.3,-1.1) to [out=180,in=250]
(-1.4,-1) ;

\fill[green!20] (-0.04,0.2) to [out=-90,in=180] (0.1,0.1) to [out=0,in=180] (2.9,0.1) to [out=0,in=-90] (3,0.2) -- (3,-0.2) to [out=100,in=0] (2.9,-0.1)
to [out=180,in=0] (0.1,-0.1) to [out=180,in=80] (0.02,-0.2);

\draw[red] (-0.04,0.2) to [out=-90,in=180] (0.1,0.1) to [out=0,in=180] (2.9,0.1) to [out=0,in=-90] (3,0.2);
\draw[red] (0.02,-0.2) to [out=80,in=180] (0.1,-0.1) to [out=0,in=180] (2.9,-0.1) to [out=0,in=100] (3,-0.2);
\draw[red] (-1.4,-1)   to [out=250,in=180] (-1.3,-1.1) to [out=0,in=180] 
(0.3,-1.1) to [out=0,in=-90] (0.4,-1);
\draw[red] (-1.8,-1.35)   to [out=30,in=180] (-1.5,-1.25) to [out=0,in=180] 
(0.5,-1.25) to [out=0,in=170] (0.85,-1.35);

\fill[orange] (-2.5,-1) circle [radius=4pt];
\fill[orange] (0,1.5) circle [radius=4pt];
\fill[orange] (1.5,-1) circle [radius=4pt];

\draw[orange] (-2.5,-1) to [out=0,in=270] (-1.5,0)  to [out=90,in=180] (-0,1.5);
\draw[orange] (0,1.5) to [out=0,in=90] (1.5,-1);

\draw[black] (-2.5,-1) node[left=0pt] {$w$};
\draw[black] (1.5,-1) node[above=1pt,right=-1pt] {$v$};

\draw (-6,0) -- (6,0);

\end{tikzpicture}
\end{center}
  \end{minipage}%
  \begin{minipage}[t] {.5\linewidth}
  \begin{center}
\begin{tikzpicture}[thick,scale=0.5]

%support
\fill[yellow!50] (-0.2,0) arc (180:360:1.7cm) -- (-0.2,0);
\fill[yellow!50] (-0.2,0) -- (-0.2,0.3) -- (3.2,0.3) -- (3.2,0); 
\fill[yellow!50] (-1.75,-1.35) -- (0.75,-1.35) -- (0.4,-0.95) -- (-1.4,-0.95);

\fill[yellow!50] (-1.4,-0.9)  to [out=100,in=270] (-1.4,0) to [out = 90, in=-180]
(0,1) to [out=0, in=90] (1.4,0.3)
-- (1.75,0.3) to [out=90,in=0] (0,1.35) to [out=180, in=90] (-1.75,0) to [out=-90, in=90] (-1.75,-0.8) to 
[out=270,in = 90] (-1.75,-1.4);
;

%Halbkurve -5 zu -4
\draw[blue] (-5,0) arc (0:35:5.5cm);
\draw[blue] (-4,0) arc (180:35:5.5cm);

\draw[blue] (5,0) -- (360:5cm) arc (360:180:5cm);

\draw[blue] (5,0) arc (0:180:4cm);

\draw[blue] (3,0) arc (0:180:2.5cm);

\draw[blue] (3,0) arc (0:-180:1.5cm);

\draw[blue] (0,0) arc (0:180:0.5cm);

\draw[blue] (-1,0) arc (0:-180:1.5cm);

\draw[blue] (-2,0) arc (0:-180:0.5cm);

\draw[thin] (-6,-6) -- (-6,6);
\draw[thin] (6,-6) -- (6,6);

\draw (-4,2) node[above=2pt, blue] {$L$};
\draw (4,0) node[above=0pt] {$L_0$};
\draw (-0.3,-1.1) node[below=0pt, red] {$\phi(L)$};

\draw[red] (-0.04,0.2) to [out=-90,in=180] (0.1,0.1) to [out=0,in=180] (2.9,0.1) to [out=0,in=-90] (3,0.2);
\draw[red] (0.02,-0.2) to [out=80,in=180] (0.1,-0.1) to [out=0,in=180] (2.9,-0.1) to [out=0,in=100] (3,-0.2);
\draw[red] (-1.4,-1)   to [out=250,in=180] (-1.3,-1.1) to [out=0,in=180] 
(0.3,-1.1) to [out=0,in=-90] (0.4,-1);
\draw[red] (-1.8,-1.35)   to [out=30,in=180] (-1.5,-1.25) to [out=0,in=180] 
(0.5,-1.25) to [out=0,in=170] (0.85,-1.35);

\draw (-6,0) -- (6,0);
\end{tikzpicture}
\end{center}
% \subcaption{Chart.}
  \end{minipage}
  \captionsetup{format=hang}
   \caption{The procedure of deleting a leaf. The yellow region sketches the support of $\phi$.}
    \label{fig:deletion}
\end{figure}

\begin{rem}
    Khanevsky applied this process only to leaves that are at distance $\geq 2$ from the root,
    by moving weight to the vertex that is at distance $2$ and closer to the root.
    This is of no significance to us.
    We allow deletion of a leaf different from the root, by moving its weight to any other leaf at distance two. The construction of $\phi$ works the same.
\end{rem}
\subsection{Effect on the Floer complex}
We denote the set of intersection points of $L$ and $L_0$ by $P$ and the Floer complex by $C:= \CF(L_0,L)$.
Suppose we obtain $L'$ by deleting a leaf from $\bar{q}$ to $\bar{p}$
as explained in Proposition \ref{lem:khanevsky} by $L' = \phi(L)$. Then the set of intersection points
of $L_0$ and $L'$ is $P'=P\backslash \{ \bar{q}, \bar{p}\}$ and we denote the Floer complex of the pair $(L_0,L')$ by $C':=\CF(L_0,L')$.
Following \cite[Appendix C]{SilvaRobbinSalamon} the chain complex $C'$ can be expressed in terms of $C$ as follows.
Recall that the differentials in $C$ and $C'$ are given by the formulas
\[
  \partial (q) = \sum\limits_{p\in P} n(q,p)p, \qquad \partial (q) = \sum_{p'\in P'} n'(q',p')p'
\]
where $n(q,p)\in \mathbb{Z}_2$ denotes the mod $2$ count of smooth lunes
for $(L_0,L)$ from $q$ to $p$ and $n'(q',p')\in \mathbb{Z}_2$ denotes the mod $2$ count of smooth lunes for $(L_0,L')$ from $q'$ to $p'$.
$n'$ can be computed from $n$ by
\[
    n'(q',p')=n(q',p') + n(q',\bar{p})n(\bar{q},p').
\]
Indeed, the only change happens when $n(q',\bar{p})=n(\bar{q},p')=1$.
Figure \ref{fig:effect_differential} below illustrates the situation.
\begin{figure}[hbt] \label{fig:0}

\begin{minipage}[t]{.5\linewidth}
    \begin{center}

\begin{tikzpicture}[thick,scale=0.5]

%%%%%%% Half Disc
\draw (-6,0) -- (6,0);
\draw[blue] (-4.5,0) arc (180:0:1.5cm);
\draw[blue] (-1.5,0) arc (180:360:1.5cm);
\draw[blue] (1.5,0) arc (180:0:1.5cm);
\draw[blue] (4.5,0) arc (180:270:1.5cm);
\draw[blue] (-4.5,0) arc (360:270:1.5cm);

\fill[blue] (-4.5,0) circle [radius=4pt];
\fill[blue] (-1.5,0) circle [radius=4pt];
\fill[blue] (1.5,0) circle [radius=4pt];
\fill[blue] (4.5,0) circle [radius=4pt];

\draw[blue] (-4.5,0) node[below=6pt,right=0pt] {$q'$};
\draw[blue] (-1.5,0) node[below=6pt,left=0pt] {$\bar{p}$};
\draw[blue] (1.5,0) node[below=6pt,right=0pt] {$\bar{q}$};
\draw[blue] (4.5,0) node[below=6pt,left=0pt] {$p'$};

\draw[white] (0,-2.8) node {};

\end{tikzpicture}
        
    \end{center}
  \end{minipage}%
  \begin{minipage}[t] {.5\linewidth}
   \begin{center}
\begin{tikzpicture}[thick,scale=0.5]
%%%%%%% Half Disc
\draw (-6,0) -- (6,0);
\draw[blue] (-4.5,0) arc (180:7:1.5cm);
\draw[blue] (4.5,0) arc (0:173:1.5cm);
\draw[blue] (4.5,0) to [out=270,in=180] (6,-3);
\draw[blue] (-4.5,0) to [out=270,in=0] (-6,-3);
\draw[blue] (-1.52,0.2) to [out=-30,in=180] (-1.4,0.1) to [out=0,in=180] (1.4,0.1) to [out=0,in=210] (1.52,0.2);

\fill[blue] (-4.5,0) circle [radius=4pt];
\fill[blue] (4.5,0) circle [radius=4pt];

\draw[blue] (-4.5,0) node[below=6pt,right=0pt] {$q'$};
\draw[blue] (4.5,0) node[below=6pt,left=0pt] {$p'$};

\end{tikzpicture}
        
    \end{center}
% \subcaption{Chart.}
  \end{minipage}
 \caption{A new lune occurs after deletion of the leaf from $\bar{q}$
    to $\bar{p}$.}
    \label{fig:effect_differential}
\end{figure}
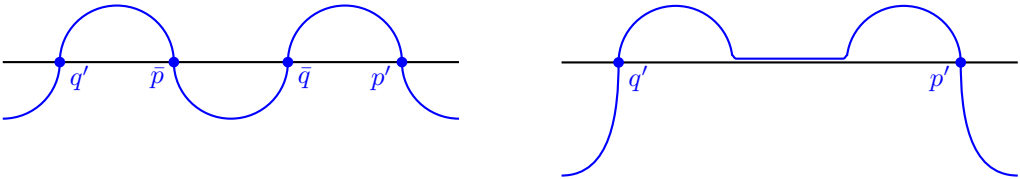

There are chain maps 
\[
\Psi \colon C \longrightarrow C' \text{\quad   and \quad  } \Phi\colon C' \longrightarrow C
\]
given by 
\[ \Psi(q) = 
\begin{cases}
  q \qquad \qquad \qquad \quad  \, \, \, \, q \neq \bar{q}, \bar{p},\\
  0 \qquad \qquad \qquad \quad \, \, \, \, q =\bar{q}, \\
  \sum\limits_{p' \in P'} n(\bar{q},p')p' \qquad q=\bar{p}
\end{cases}
\]
and
\[ \Phi(q') = q' + n(q',\bar{p})\bar{q}.
\]
These chain maps are chain homotopy inverses to each other:
$\Psi \circ \Phi = \mathrm{Id}$ and $\Phi \circ \Psi - \mathrm{Id} = \partial T + T \partial$ for the chain homotopy $T \colon C \longrightarrow C$ given by
\begin{align*}
 T(q) = 
    \begin{cases}
      \overline{q} \qquad q = \overline{p}, \\
       0  \qquad q\neq \overline{p}.  
    \end{cases}
\end{align*}
\subsection{Effect on the filtration}\label{subsec:interleaving}
We first study the change of the graph $T = T(L)$ to the graph $T' = T(L')$.
There are two cases, one for a leaf above $L_0$ (Case 1) and one for a leaf below $L_0$ (Case 2), as shown in Figure \ref{fig:trees_case1}.
In Case 1, let $j\in \{1, \dots, 2n\}$ be the index with $\bar{q}= s_j$ and $\bar{p}= s_{j+1}$.
Similarly in case 2, let $j\in \{1, \dots, 2n\}$ be the index with $\bar{p}= s_j$ and $\bar{q}= s_{j+1}$.
The intersections point of $L_0$ with $L'$ are $s_1, \dots, s_{j-1},s_{j+2}, \dots s_{2n}$.
For the ordering of the edges of $T(L')$ we use the numbering by $1, \dots, j-1, j+2, \dots, 2n$.
Let $k$ be the index of the edge adjacent to $w$ that lies on the path from $w$ to $v$.
By renumbering if needed, we may arrange that $j+2 \leq k \leq 2n$.
Figure \ref{fig:change_of_trees} illustrates the change of the graphs in Case 1.
Case 2 is analogous with reversed arrows.

\begin{figure}[hbt] \label{fig:0}

\begin{minipage}[t]{.5\linewidth}
    \begin{center}

\begin{tikzpicture}[thick,scale=0.5]
%oben grün

\draw (-6,0) -- (6,0);

\draw (-6,4) node[right=0pt] {Case 1};

\draw[blue] (-4,0) to [out=90, in =180] (-3,1.5) to [out=0,in=90] (-2,0);
\draw[blue, dotted] (-4,0) to [out=-90,in=0] (-5,-2) to [out=180,in=-90] (-6,0);
\draw[blue, dotted] (-2,0) to [out=-90,in=180] (-1,-2) to [out=0,in=-90] (0,0);

\fill[blue] (-6,0) circle [radius=4pt];
\draw[blue] (-6,0) node[below=5pt,left=-3pt] {$s_{j-1}$};
\fill[blue] (-4,0) circle [radius=4pt];
\draw[blue] (-4,0) node[below=5pt,left=-3pt] {$s_j$};
\fill[blue] (-2,0) circle [radius=4pt];
\draw[blue] (-2,0) node[below=5pt,right=-3pt] {$s_{j+1}$};
\fill[blue] (0,0) circle [radius=4pt];
\draw[blue] (0,0) node[below=5pt,right=-3pt] {$s_{j+2}$};

\fill[blue] (4,0) circle [radius=4pt];
\draw[blue] (4,0) node[below=5pt,left=-3pt] {$s_{k}$};
\fill[blue] (6,0) circle [radius=4pt];
\draw[blue] (6,0) node[below=5pt,left=-10pt] {$s_{k+1}$};

\fill[green] (-3,0.75) circle [radius=4pt];
\draw[black] (-3,0.75) node[right=0pt] {$v$};
\fill[green] (-1,-3) circle [radius=4pt];
\fill[green] (5,1) circle [radius=4pt];
\draw[black] (5,1) node[right=0pt] {$w$};
\draw[green] (-3,0.75) to [out=-90,in=180] (-1,-3);
\draw[green] (-1,-3) to [out=0,in=225] (4,-2) to [out=45,in=-90] (5,1);

\fill[orange] (-5,-1) circle [radius=4pt];
\fill[orange] (-3,3) circle [radius=4pt];
\fill[orange] (-1,-1) circle [radius=4pt];
\draw[orange] (-5,-1) to [out=90,in=180] (-3,3);
\draw[orange] (-3,3) to [out=0,in=90] (-1,-1);

\end{tikzpicture}
        
    \end{center}
  \end{minipage}%
  \begin{minipage}[t] {.5\linewidth}
   \begin{center}
\begin{tikzpicture}[thick,scale=0.5]

\draw (-6,0) -- (6,0);

\draw (-6,4) node[right=0pt] {Case 2};
\draw (-6,-3.2) node[white] {};

\draw[blue] (-4,0) to [out=-90, in =180] (-3,-1.5) to [out=0,in=-90] (-2,0);
\draw[blue, dotted] (-4,0) to [out=90,in=0] (-5,2) to [out=180,in=90] (-6,0);
\draw[blue, dotted] (-2,0) to [out=90,in=180] (-1,2) to [out=0,in=90] (0,0);

\fill[blue] (-6,0) circle [radius=4pt];
\draw[blue] (-6,0) node[below=5pt,left=-14pt] {$s_{j-1}$};
\fill[blue] (-4,0) circle [radius=4pt];
\draw[blue] (-4,0) node[below=5pt,left=-3pt] {$s_j$};
\fill[blue] (-2,0) circle [radius=4pt];
\draw[blue] (-2,0) node[below=5pt,right=-3pt] {$s_{j+1}$};
\fill[blue] (0,0) circle [radius=4pt];
\draw[blue] (0,0) node[below=5pt,right=-3pt] {$s_{j+2}$};

\fill[blue] (4,0) circle [radius=4pt];
\draw[blue] (4,0) node[below=5pt,left=-3pt] {$s_{k}$};
\fill[blue] (6,0) circle [radius=4pt];
\draw[blue] (6,0) node[below=5pt,right=-3pt] {$s_{k+1}$};

\fill[green] (-3,-0.75) circle [radius=4pt];
\draw[black] (-3,-0.75) node[right=0pt] {$v$};
\fill[green] (-1,3) circle [radius=4pt];
\fill[green] (5,-1) circle [radius=4pt];
\draw[black] (5,-1) node[right=0pt] {$w$};
\draw[green] (-3,-0.75) to [out=90,in=180] (-1,3);
\draw[green] (-1,3) to [out=0,in=135] (4,2) to [out=-45,in=90] (5,-1);

\fill[orange] (-5,1) circle [radius=4pt];
\fill[orange] (-3,-3) circle [radius=4pt];
\fill[orange] (-1,1) circle [radius=4pt];
\draw[orange] (-5,1) to [out=-90,in=180] (-3,-3);
\draw[orange] (-3,-3) to [out=0,in=-90] (-1,1);

\end{tikzpicture}
        
    \end{center}
% \subcaption{Chart.}
  \end{minipage}
  \caption{Two scenarios for the leaf.}
    \label{fig:trees_case1}
\end{figure}
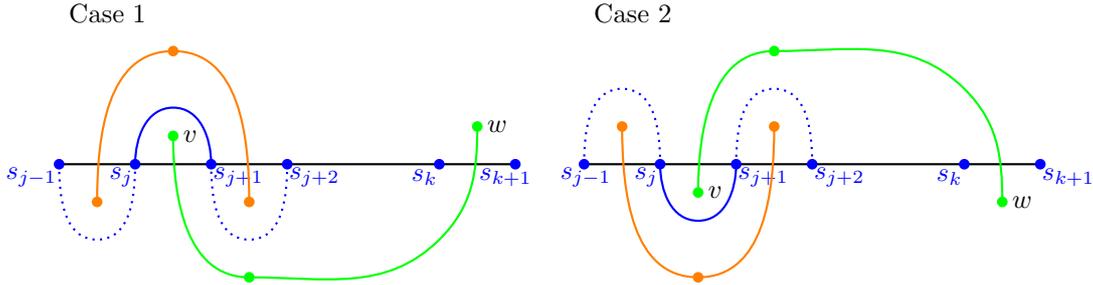
\begin{figure}[hbt] \label{fig:0}

\begin{minipage}[t]{.5\linewidth}
 \begin{center}

\begin{tikzpicture}[thick,scale=0.65]

\draw (-2,-1) node[right=0pt] {$T(L)$};

\draw[green]  (-6,-4) circle [radius=0.7cm];
\draw[green] (-6,-4) node {$S_0$};
\fill[green] (-3,-4) circle [radius=4pt];
\draw[green] (-3,-4) node[above=0pt] {$a_2$};
\fill[green] (0,-4) circle [radius=4pt];
\draw[green] (0,-4) node[above=0pt] {$a_1$};
\fill[green] (3,-4) circle [radius=4pt];
\draw[green] (3,-4) node[above=0pt] {$a(v)$};
\draw[green]  (0,-6) circle [radius=0.7cm];
\draw[green] (0,-6) node {$S_1$};

\draw[green] (-5.3,-4) -- (-3,-4);
\draw[green,middlearrow={stealth}] (-3,-4) -- (0,-4);
\draw[blue] (-1.5,-4) node[above=0pt] {$k$};
\draw[green,middlearrow={stealth reversed}] (0,-4) -- (3,-4);
\draw[blue] (1.5,-4) node [above=0pt] {$j$};
\draw[green] (0,-4) -- (0,-5.3);

\draw[orange]  (-6,-10) circle [radius=0.7cm];
\draw[orange] (-6,-10) node {$T_0$};
\fill[orange] (-3,-10) circle [radius=4pt];
\draw[orange] (-3,-10) node[above=0pt] {$b$};
\fill[orange] (0,-9) circle [radius=4pt];
\draw[orange] (0,-9) node[above=0pt] {$b_1$};
\fill[orange] (0,-11) circle [radius=4pt];
\draw[orange] (0,-11) node[below=0pt] {$b_2$};
\draw[orange]  (3,-9) circle [radius=0.7cm];
\draw[orange] (3,-9) node {$T_1$};
\draw[orange]  (3,-11) circle [radius=0.7cm];
\draw[orange] (3,-11) node {$T_2$};

\draw[orange] (-5.3,-10) -- (-3,-10);
\draw[orange, middlearrow={stealth}] (-3,-10) -- (0,-9);
\draw[blue] (-1.5,-9.5) node [above=0pt] {$j-1$};
\draw[orange,middlearrow={stealth}] (-3,-10) -- (0,-11);
\draw[blue] (-1.5,-10.5) node [below=0pt] {$j+1$};
\draw[orange] (0,-9) -- (2.3,-9);
\draw[orange] (0,-11) -- (2.3,-11);

\end{tikzpicture}
\end{center}
  \end{minipage}%
 \hfill %
  \begin{minipage}[t] {.5\linewidth}
   \begin{center}
\begin{tikzpicture}[thick,scale=0.65]
\draw (-2,-1) node[right=0pt] {$T(L')$};

\draw[green]  (-6,-4) circle [radius=0.7cm];
\draw[green] (-6,-4) node {$S_0$};
\fill[green] (-3,-4) circle [radius=4pt];
\draw[green] (-3,-4) node[above=6pt] {$a_2+ a(v)$};
\draw[green] (-3,-4) node[above=0pt] {$+\epsilon$};
\fill[green] (0,-4) circle [radius=4pt];
\draw[green] (0,-4) node[above=0pt] {$a_1 - \epsilon$};
\draw[green]  (0,-6) circle [radius=0.7cm];
\draw[green] (0,-6) node {$S_1$};

\draw[green] (-5.3,-4) -- (-3,-4);
\draw[green,middlearrow={stealth}] (-3,-4) -- (0,-4);
\draw[blue] (-1.5,-4) node[below=0pt] {$k$};
\draw[green] (0,-4) -- (0,-5.3);

\draw[orange]  (-6,-10) circle [radius=0.7cm];
\draw[orange] (-6,-10) node {$T_0$};
\fill[orange] (-3,-10) circle [radius=4pt];
\draw[orange] (-3,-10) node[above=0pt] {$b-\epsilon$};
\fill[orange] (0,-10) circle [radius=4pt];
\draw[orange] (0,-10) node[above=5pt] {$b_1+b_2$};
\draw[orange] (0,-10) node[above=0pt] {$+\epsilon$};
\draw[orange]  (3,-9) circle [radius=0.7cm];
\draw[orange] (3,-9) node {$T_1$};
\draw[orange]  (3,-11) circle [radius=0.7cm];
\draw[orange] (3,-11) node {$T_2$};

\draw[orange] (-5.3,-10) -- (-3,-10);
\draw[orange, middlearrow={stealth}] (-3,-10) -- (0,-10);
\draw[blue] (-1.5,-10) node [below=0pt] {$j-1$};
\draw[orange] (0,-10) -- (2.3,-9);
\draw[orange] (0,-10) -- (2.3,-11);

\end{tikzpicture}
\end{center}
% \subcaption{Chart.}
  \end{minipage}
  \captionsetup{format=hang}
  \caption{Change of trees. The only changes of weights happen at the vertices,
    where the weights are recorded.}
\label{fig:change_of_trees}
\end{figure}
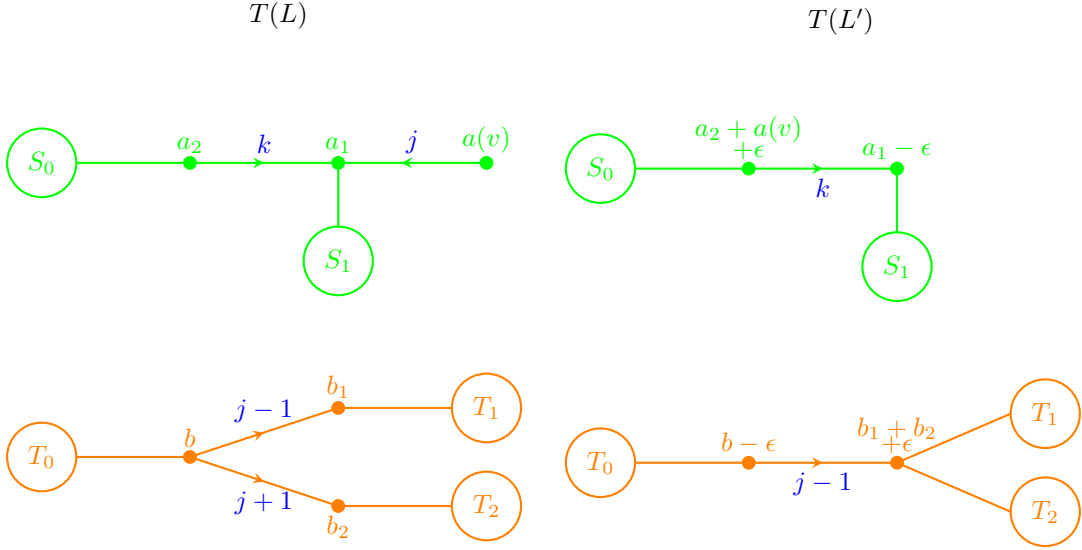

Let $\mathcal{A}$ and $\mathcal{A}'$ denote the action functionals for $C$ and $C'$.
Recall that they are defined up to an additive constant.
\begin{lem}\label{lem:action}
    After possibly adding an overall constant to $\mathcal{A}'$, the two action functionals $\mathcal{A}$ and $\mathcal{A}'$ are related as follows:
    in Case 1
    \begin{align*}
     \mathcal{A}'(s_l) =
     \begin{cases}
        \mathcal{A}(s_l)+ \frac{a(v)+\epsilon}{2} \qquad j+2\leq l \leq k,\\
        \mathcal{A}(s_l)- \frac{a(v)+\epsilon}{2}\qquad \text{else},
     \end{cases}
   \end{align*}
    and in Case 2
    \begin{align*}
     \mathcal{A}'(s_l) =
     \begin{cases}
        \mathcal{A}(s_l)- \frac{a(v)+\epsilon}{2} \qquad j+2\leq l \leq k,\\
        \mathcal{A}(s_l) + \frac{a(v)+\epsilon}{2} \qquad \text{else}.
     \end{cases}
   \end{align*}
\end{lem}

The rest of this section is devoted to a proof of this.

There is always a smooth lune between $s_i$ and $s_{i+1}$:
To see this, we introduce some notation.
Given an edge $e$ in $T$ we denote by $v(e)$ the vertex of the edge that is further away from the root. The subtree $T_v$ is the subtree of descendents of the vertex $v$.
Consider now the edge $e_i$ in $T$ and the subtree
$T_{v(e_i)}$. The corresponding region in $\Sigma$ is an embedded smooth lune $u^i$
between $s_i$ and $s_{i+1}$. If $e_i$ points away from the root, it is a lune from $s_{i+1}$ to $s_i$.
If $e_i$ points towards the root, it is a lune from $s_i$ to $s_{i+1}$.

The smooth lune $u^i$ has area
$W(T_{v(e_i)})$, where $W$ denotes the total weight of a weighted tree.
From equation (\ref{eq:action_area}) we deduce the following formula for the action difference between two neighboured intersection points:
\begin{align}\label{eq:action_diff_tree}
    \mathcal{A}(s_{j+1}) - \mathcal{A}(s_j) 
    = s(e_j) W(T_{v(e_j)}),
\end{align}
where the sign is determined as follows:
\begin{align*}
    s(e) = 
    \begin{cases}
      1 \qquad \, \; \; \text{ edge points away from the root} \\
      -1 \qquad \text{ edge points towards the root.}
    \end{cases}
\end{align*}

Using this formalism we can now prove Lemma \ref{lem:action}.
\begin{proof}[Proof of Lemma \ref{lem:action}]
   By Remark \ref{rem:action} it is enough to prove that $\mathcal{A}'$ defined by the formula
    in Lemma \ref{lem:action} satisfies (\ref{eq:action_diff_tree})
    for $T'$.
    We consider Case $1$. First note that $s(e_l)=s(e'_l)$ for all $l\neq j,j+1$.
    For $l \notin \{j-1,j,j+1,k\}$ we compute
\begin{align*}
    \mathcal{A}'(s_{l+1}) -\mathcal{A}'(s_{l}) &=
    \mathcal{A}(s_{l+1}) -\mathcal{A}(s_{l})\\
    &=s(e_l)W\left(T_{v(e_l)}\right) \\
    &=s(e_l')W\left(T'_{v(e_l)}\right).
\end{align*}
For $l=j-1$ first note that $s(e_{j-1})=s(e_{j+1})=s(e'_{j-1})=1$ and as it can be seen from Figure \ref{fig:change_of_trees} we have
\[
    W\left(T'_{v(e'_{j-1})}\right) = W\left(T_{v(e_{j-1})}\right) + W\left(T_{v(e_{j+1})}\right) + \epsilon.
\]
Therefore
\begin{align*}
    \mathcal{A}'(s_{j+2}) -\mathcal{A}'(s_{j-1})
        &= \left(\mathcal{A}(s_{j+2}) + \frac{a(v)+\epsilon}{2} \right)
        -\left( \mathcal{A}(s_{j-1})
        - \frac{a(v)+\epsilon}{2} \right)\\
         &= \mathcal{A}(s_{j}) -\mathcal{A}(s_{j-1})
    + \mathcal{A}(s_{j+2}) -\mathcal{A}(s_{j+1}) + \epsilon\\
      &=s(e_{j-1})W\left(T_{v(e_{j-1})}\right)
    +s(e_{j+1})W\left(T_{v(e_{j+1})}\right) + \epsilon\\
     &=s(e_{j-1}')W\left(T'_{v(e'_{j-1})}\right).
\end{align*}
Similarly for $l=k$ we see from Figure \ref{fig:change_of_trees} that $s(e_k)=s(e'_k)=1$
and
\[
    W\left(T'_{v(e'_k)}\right) = W\left(T_{v(e_k)}\right) -a(v) - \epsilon.
\]
Therefore
\begin{align*}
    \mathcal{A}'(s_{k+1}) -\mathcal{A}'(s_{k})
     &= \left( \mathcal{A}(s_{k+1})- \frac{a(v)+\epsilon}{2} \right)
     - \left( \mathcal{A}(s_{k}) + \frac{a(v)+\epsilon}{2} \right)\\
     &= s(e_k)W\left(T_{v(e_k)}\right) -a(v)- \epsilon \\
     &=s(e_k')W\left(T'_{v(e'_k)}\right).
\end{align*}
This shows that $\mathcal{A}'$ is the action functional for $(L_0,L')$. 
Case $2$ works similarly.
\end{proof}

\subsection{Effect on the barcode}
The chain maps $\Phi$ and $\Psi$ are related to the filtrations on $C$ and $C'$ as follows.
\begin{prop}\label{prop:interleaving}
    The chain maps $\Phi$ and $\Psi$ shift action by at most $\frac{a(v)+\epsilon}{2}$.
    The chain homotopy $T$ shifts action by at most $a(v)+\epsilon$.
\end{prop}
\noindent
It follows that 
\[
    \Phi_* \colon \HF^{\leq \alpha}(L_0,L') \longrightarrow \HF^{\leq \alpha + \frac{a(v)+ \epsilon}{2}}(L_0,L)
\]
and
\[
    \Psi_* \colon \HF^{\leq \alpha}(L_0,L) \longrightarrow \HF^{\leq \alpha + \frac{a(v)+ \epsilon}{2}}(L_0,L')
\]
satisfy
\[
    \Psi_* \circ \Phi_* \colon = i'_{\alpha, \alpha + a(v)+ \epsilon} \colon
    \HF^{\leq \alpha}(L_0,L') \longrightarrow \HF^{\leq \alpha +a(v)+ \epsilon}(L_0,L')
\]
and
\[
    \Phi_* \circ \Psi_* \colon = i_{\alpha, \alpha + a(v)+ \epsilon} \colon
    \HF^{\leq \alpha}(L_0,L) \longrightarrow \HF^{\leq \alpha +a(v)+ \epsilon}(L_0,L).
\]
This is called an $\frac{a(v)+ \epsilon}{2}$-interleaving in the theory of persistence modules.
The bounds (\ref{ineq:beta}) and (\ref{ineq:gamma}) now follow from general persistence theory
\cite{UsherZhang}.
For the convenience of the reader, we include an outline of the arguments
following closely \cite[Section 4.2]{PolterovichRosenSamvelyanZhang}.

We abbreviate $\mathcal{B}= \mathcal{B}(L_0,L)$, $\mathcal{B}' = \mathcal{B}(L_0,L')$, denote by $\beta_i$, $\beta_i'$ the lengths of the finite bars in $\mathcal{B}$ and $\mathcal{B}'$ and by $\gamma$ and $\gamma'$ the spectral metrics.
The algebraic stability theorem implies that there exists an \textit{$\frac{a(v)+ \epsilon}{2}$-matching}
of the barcodes $\mathcal{B}= \mathcal{B}(L_0,L)$ and
$\mathcal{B}' = \mathcal{B}(L_0,L')$.
\footnote{The algebraic stability theorem is part of the isometry theorem, which states
that the interleaving distance on persistence modules and the bottleneck distance
on their barcodes coincide.}
This is a bijection $\mu \colon \mathcal{B}_0 \longrightarrow \mathcal{B}_0'$
of subsets $\mathcal{B}_0 \subseteq \mathcal{B}$ and $\mathcal{B}_0' \subseteq \mathcal{B}'$
satisfying
\begin{enumerate}
    \item $\mathcal{B}_0$ contains all bars from $\mathcal{B}$ of length $>a(v)+ \epsilon$,
    \item $\mathcal{B}_0'$ contains all bars from $\mathcal{B'}$ of length $>a(v)+ \epsilon$,
    \item If $\mu([a,b)) = [a',b')$ then $\vert a - a'\vert \leq \frac{a(v)+ \epsilon}{2}$
    and $\vert b-b'\vert \leq \frac{a(v)+ \epsilon}{2}$.
\end{enumerate}
Let $[c_1,\infty) \supseteq [c_2,\infty)$ be the two infinite bars of $\mathcal{B}$ 
and $[c_1',\infty) \supseteq [c_2',\infty)$ be the two infinite bars of $\mathcal{B}'$.
Then either $\mu([c_1,\infty)) = [c_1',\infty)$ and $\mu([c_2,\infty)) = [c_2',\infty)$
or
$\mu([c_1,\infty)) = [c_2',\infty)$ and $\mu([c_2,\infty)) = [c_1',\infty)$.
In the first case, one immediately sees
\[
     \gamma - \gamma'  = (c_2 - c_1) - (c_2' - c_1') = (c_2 - c_2') +(c_1' - c_1)  \leq \frac{a(v)+ \epsilon}{2} + \frac{a(v)+ \epsilon}{2} =a(v)+ \epsilon.
\]
For the second case, one computes similarly
\[
    \gamma - \gamma'  = (c_2 - c_1) - (c_2' - c_1')
     = (c_2 - c_2') + (c_1' - c_1)  \leq (c_2 - c_1') + (c_2' - c_1) \leq a(v)+ \epsilon.
\]
By symmetry it follows that $\vert \gamma - \gamma'\vert < a(v)+ \epsilon$.

The argument is similar for the finite bars.
We need to show that $\beta_k - \beta_k' \leq a(v)+ \epsilon$. 
If $\beta_k \leq a(v)+ \epsilon$ there is nothing to show.
If $\beta_k > a(v)+ \epsilon$, consider the bars $b_1, \dots, b_k$ corresponding to 
$\beta_1, \dots, \beta_k$. All of them are matched by $\mu$ to some bars
$\mu(b_1), \dots, \mu(b_k)$ from $\mathcal{B}'$ with
lengths $\alpha_1', \dots, \alpha_k'$. By the third property
of a matching, we have $\vert \beta_i - \alpha_i' \vert \leq a(v)+ \epsilon$
for $i \in \{1, \dots, k\}$.
Let $\sigma \colon \{ 1, \dots, k\} \longrightarrow \{ 1, \dots, k\}$ be the permutation that
orders the bars according to their lengths:
$\alpha_{\sigma(1)}' \geq \dots \geq \alpha_{\sigma(k)}'$.
Then
\[
    \max_i \vert \beta_i - \alpha_{\sigma(i)}' \vert \leq  \max_i \vert \beta_i - \alpha_i' \vert
    \leq a(v)+ \epsilon,
\]
where the first inequality comes from the fact, that the \textit{best matching} of a sequence of
numbers is the monotone one (\cite[Lemma 4.1.1]{PolterovichRosenSamvelyanZhang}).
It follows that
\begin{align*}
    a(v)+ \epsilon \geq \vert \beta_k - \alpha_k' \vert \geq \vert \beta_k - \alpha_k' \vert 
    \geq \beta_k - \alpha_k' \geq \beta_k - \beta_k'.
\end{align*}
By symmetry we conclude that $\vert \beta_k - \beta_k'\vert \leq a(v)+ \epsilon$.
This shows how Proposition \ref{prop:interleaving} implies the inequalities
(\ref{ineq:beta}) and (\ref{ineq:gamma}).

For the proof of Proposition \ref{prop:interleaving} we need the following
\begin{lem} \label{lem:A}
Assume Case 1. Then for any two intersection points $q, p \in P\backslash
\{\bar{q}, \bar{p}\}$
\begin{enumerate}
    \item $n(q,\bar{p}) = 1$ implies $q=s_i$ for $j+2 \leq i \leq k$.
    \item $n(\bar{q},p) = 1$ implies $p=s_i$ for $k+1 \leq i \leq j-1$.
\end{enumerate}
\end{lem}
\begin{proof}
    Let $\bar{q} \neq s_l\in P$ be the intersection point that is
    a neighbour of $\bar{p}$ viewed on $L$.
    Then $j+1 \leq l \leq k$.
    By Lemma \ref{lem:lune} any lune $u$ entering $\bar{p}$ \textit{from the right} must have
    $s_l \notin u(\mathbb{D} \cap \R) = [\bar{p}, s_i]$. In particular,
    $j+2 \leq i  \leq l \leq k$. This shows (1).
    The proof for part (2) works similarly.
\end{proof}

\begin{proof}[Proof of Proposition \ref{prop:interleaving}.]
    We only show it for Case 1. The other case is similar. 
Let $q = s_i$ for some $i\neq j, j+1$.
Suppose $n(q,\overline{p}) = 1$. Then
by part (1) of Lemma \ref{lem:A}
$q = s_i$ for some $j+2 \leq i \leq k$. Hence $\mathcal{A}(q) = \mathcal{A}'(q) - \frac{a(v) + \epsilon}{2}$. Therefore for any $q \in P$

\begin{align*}
    \mathcal{A}(\Phi q ) &= \mathcal{A}(q + n(q,\overline{p}) \overline{q}) \\
    &=
    \begin{cases}
        \mathcal{A}(q) \qquad \qquad \qquad\qquad \qquad \qquad \qquad \text{if } n(q,\overline{p}) = 0\\
        \max \{ \mathcal{A}(q), \mathcal{A}(\overline{q})\} \, \, \, \,\qquad
        \qquad \qquad \qquad \text{if } n(q,\overline{p} \neq 0
    \end{cases}\\
    &\leq 
    \begin{cases}
        \mathcal{A}'(q) + \frac{a(v) + \epsilon}{2} \, \, \, \, \, \qquad \qquad \qquad \qquad \qquad \text{if } n(q,\overline{p}) = 0\\
        \max\{\mathcal{A}'(q)+ \frac{a(v) + \epsilon}{2}, \mathcal{A}(q)+a(v) + \epsilon\} \qquad \, \text{if } n(q,\overline{p}) \neq 0
     \end{cases}\\
     &= \mathcal{A}'(q) + \frac{a(v) + \epsilon}{2}.
\end{align*}
This shows that $\Phi$ shifts action by at most $\frac{a(v)+\epsilon}{2}$.
Similarly, $\Psi$ shifts action by at most $\frac{a(v) + \epsilon}{2}$: 
\begin{align*}
    \mathcal{A}'(\Psi \overline{p}) &= \max \{\mathcal{A}'(p) \, \vert \, n(\overline{q},p) \neq 0 \} \\
    &\leq \max \left \{\mathcal{A}(p) - \frac{a(v) + \epsilon}{2}\, \big \vert \,n(\overline{q},p) \neq 0 \right\} \\
    &\leq \mathcal{A}(\overline{q}) - \frac{a(v) + \epsilon}{2} \\
    & = \mathcal{A}(\overline{p}) + \frac{a(v) + \epsilon}{2},
\end{align*}
where the first inequality follows from part (2) of Lemma \ref{lem:A}.

Clearly, the chain homotopy $T$ shifts action by at most $a(v) + \epsilon$.
\end{proof}

%%%%%%%%%%%%%%%%%%% Section 3 - The smallest bar %%%%%%%%%%%%%%%%%%%%%%%
\section{The shortest bar.}\label{sec:smallest_bar}
The goal of this section is to show Proposition \ref{lem:main}.
As in the previous section, we only consider the pair $(L_0,L)$
for $L\in \mathcal{L}_0$. We assume that $L_0$ and $L$ intersect in at least 
$2n \geq 4$ points. Moreover, the action values $\mathcal{A}(q)$,
$q\in L_0 \cap L$, are assumed to be distinct.
Let $[a,b)$ be the shortest (finite) bar in $\mathcal{B}(L_0,L)$ and let $\bar{q},\bar{p}\in L_0 \cap L$ be the intersection points satisfying $\mathcal{A}(\bar{q}) = b$ and $\mathcal{A}(\bar{p})=a$.
The goal is to show that there is a leaf from $\bar{q}$ to $\bar{p}$.
The proof is based on the following two lemmas whose proofs will be given shortly after.
\begin{lem}\label{lem:area_of_lune}
    Let $x,y \in L_0 \cap L$ such that $n(x,y)=1$. 
    Then $\mathcal{A}(x)-\mathcal{A}(y)\geq b-a.$
\end{lem}
\begin{lem}\label{lem:leaf}
     Let $u\colon \mathbb{D} \longrightarrow \Sigma$
     be a smooth lune for $(L_0,L)$.
     Then there exists a leaf $v\colon \mathbb{D} \longrightarrow \Sigma$
     such that $\mathrm{Im}(v) \subseteq \mathrm{Im}(u)$.
\end{lem}

Let $\mathcal{S} \subset \CF(L_0,L)$ be a Jordan basis for $\partial$
(see the definition in section \ref{subsec:barcodes}). 
We explain how Proposition \ref{lem:main} follows assuming the previous two lemmas.
\begin{proof}[Proof of Proposition \ref{lem:main}]
    As a first step towards the proof, we show that $n(\bar{q},\bar{p})=1$:
Let 
$$e = \bar{q}+ \sum_{\mathcal{A}(q) < \mathcal{A}(\bar{q})} n_{q}q\in \mathcal{S}$$
be the basis element with $\mathcal{A}(e) = b$ and 
$$f= \bar{p}+ \sum_{\mathcal{A}(p) < \mathcal{A}(\bar{p})} m_{p}p\in \mathcal{S}$$ 
be the basis element with $\mathcal{A}(f)=a$. 
Since $\partial e = f$, there exists $q' \in L_0\cap L$ with $\mathcal{A}(q') \leq \mathcal{A}(\bar{q})=b$ such that $n(q',\bar{p}) =1$. By Lemma \ref{lem:area_of_lune} it follows that $\mathcal{A}(q')=\mathcal{A}(\bar{q})$. Under the assumption that all the action values are distinct, we conclude that $\bar{q} = q'$, hence $n(\bar{q},\bar{p})=n(q',\bar{p})=1$.
In particular, there exists a unique smooth lune $u_{\bar{q}, \bar{p}}$ from $\bar{q}$ to $\bar{p}$.

Suppose by contradiction that the smooth lune $u_{\bar{q}, \bar{p}}$ is not a leaf. By Lemma \ref{lem:leaf}, there exists a leaf, say from $x$ to $y$, whose image is strictly contained in $\mathrm{Im}(u_{\bar{q}, \bar{p}})$. If $n(x,y) = 1$, then the inequality
\[
    \mathcal{A}(x) - \mathcal{A}(y) < \mathcal{A}(\bar{q}) - \mathcal{A}(\bar{p}) = b-a
\]
is a contradiction to Lemma \ref{lem:area_of_lune}. If $n(x,y) = 0$, then there are two smooth lunes from $x$ to $y$, one of them being a leaf.
It follows from Lemma \ref{lem:lune} that the only possibility for such a situation is when $L_0\cap L = \{x,y\}$. This is not the case here because there are at least $4$ intersection points.
We conclude that $u_{\bar{q}, \bar{p}}$ is a leaf from $\bar{q}$ to $\bar{p}$.
\end{proof}

\begin{proof}[Proof of Lemma \ref{lem:area_of_lune}.]
    Let $x'\in L_0 \cap L$ be a generator of $CF(L_0,L)$ satisfying
    \[
       \mathcal{A}(x') = \min \{\mathcal{A}(x'') \vert n(x'',y) = 1\}.
    \]
    By minimality of $\mathcal{A}(x')$, $\partial x'$ is not a boundary in $\CF^{<\mathcal{A}(x')}(L_0,L)$. Hence $\mathcal{A}(x')$ is an upper end of a finite bar $[\tilde{a},\tilde{b})$.
    Since $n(x,y)=1$ one has
    $\tilde{b} = \mathcal{A}(x') \leq \mathcal{A}(x)$.
    Let 
    $$e' = x' + \sum_{\mathcal{A}(x'') < \mathcal{A}(x')} l_{x''}x''\in \mathcal{S}$$ 
    be the basis element with $\mathcal{A}(e') = \tilde{b}$. Then $\partial e'$ contains $y$ as a summand because $n(x',y)=1$, but $n(x'',y)=0$ for all $x''$ with $\mathcal{A}(x'')<\mathcal{A}(x')$. In particular, $\tilde{a}= \mathcal{A}(\partial e) \geq \mathcal{A}(y)$.
    By minimality of the bar $[a,b)$ it follows that 
    \[
       b- a \leq \tilde{b} - \tilde{a} \leq \mathcal{A}(x) - \mathcal{A}(y).
    \]
\end{proof}

\begin{proof}[Proof of Lemma \ref{lem:leaf}.]
Without loss of generality, we may assume that $u$ is a smooth lune from $s_1$ to $s_l$, $3 \leq l \leq 2n$ with $u(\mathbb{D}\cap \R) = [s_1,s_l]$.
Intuitively, if there were no leaf contained in $\mathrm{Im}(u)$, whenever some part of the graph $T(L)$ enters the region of the lune, it will also leave the lune. Entering and leaving happens only on $[s_1,s_l]$, along which the two components of the graph alternate. This is not possible.

To make this argument rigorous, let us assume that $\mathrm{Im}(u)$ does not contain any leaf. We denote by $C_1$ the upper, and by $C_2$ the lower component of $\Sigma \backslash L$.
Consider the preimages $R_k:= u^{-1}(C_k)\subset \mathbb{D}$ of $C_k$ under $u$ for $k=1,2$.
See Figure \ref{fig:leaves} for an illustration of these sets.
\begin{figure}[hbt]
\begin{minipage}[t]{.5\linewidth}
  \begin{center}
\begin{tikzpicture}[thick,scale=0.4]

%%%%%%% Half Disc
\fill[orange!40] (6,0) -- (180:6cm) arc (180:0:6cm) -- cycle;
\fill[green!20] (4,0) -- (-4,0) arc (180:0:4cm) -- cycle;
\fill[orange!40] (-3,0) to [out=80,in= 180] (-2.5,1) to [out=0,in=100] (-2,0);
\fill[orange!40] (-1,0) to [out=90,in= 180] (-0.5,2) to [out=0,in=90] (0,0);
\draw[blue] (6,0) -- (180:6cm) arc (180:0:6cm);
\draw (-6,0) -- (6,0);

\fill[blue] (-6,0) circle [radius=4pt];
\fill[blue] (-4,0) circle [radius=4pt];
\fill[blue] (-3,0) circle [radius=4pt];
\fill[blue] (-2,0) circle [radius=4pt];
\fill[blue] (-1,0) circle [radius=4pt];
\fill[blue] (-0,0) circle [radius=4pt];
\fill[blue] (4,0) circle [radius=4pt];
\fill[blue] (6,0) circle [radius=4pt];
\draw[blue] (-6,0) node[below=0pt] {$\lambda_1$};
\draw[blue] (-4,0) node[below=0pt] {$\lambda_2$};
\draw[blue] (-3,0) node[below=0pt] {$\lambda_3$};
\draw[blue] (-2,0) node[below=0pt] {$\lambda_4$};
\draw[blue] (-1,0) node[below=0pt] {$\lambda_5$};
\draw[blue] (0,0) node[below=0pt] {$\lambda_6$};
\draw[blue] (4,0) node[below=0pt] {$\lambda_7$};
\draw[blue] (6,0) node[below=0pt] {$\lambda_8$};

%\draw (0,0) node[below=1pt] {$\mathbb{D} \cap \R$};
%\draw (0,6) node[above=1pt, blue] {$\mathbb{D} \cap S^1$};

\draw[green] (6,-4) node[above=0pt, right=1pt] {$ $};
\draw[->] (8,2.5) -- node[scale=1,above=0pt]{$u$} (10.5,2.5);

\end{tikzpicture}
\end{center}
  \end{minipage}%
 %\hfill %
  \begin{minipage}[t] {.5\linewidth}
   \begin{center}
\begin{tikzpicture}[thick,scale=0.5]
%oben grün
\fill[green!20] (-6,6) -- (-6,0) -- (6,0) -- (6,6) --cycle;

%unten orange
\fill[orange!40] (-6,-6) -- (-6,0) -- (6,0) -- (6,-6) -- cycle;

%Halbkurve -5 zu 5
\fill[orange!40] (-5,0) -- (5,0) arc (0:180:5cm);
\draw[blue] (-5,0) arc (180:0:5cm);

%Halbkurve -4 zu -5
\fill[green!20] (6,0) -- (-4,0) arc (180:325:5.5cm); 
\draw[blue] (-4,0) arc (180:325:5.5cm);
\fill[green!20] (-6,0) -- (-5,0) arc (360:325:5.5cm); 
\draw[blue] (-5,0) arc (360:325:5.5cm);

%Halbkurve -4 zu 1
\fill[green!20] (-4,0) -- (1,0) to [out=90,in=0] (0,3.5) to [out=180,in=90] (-4,0);
\draw[blue] (1,0) to [out=90,in=0] (0,3.5) to [out=180,in=90] (-4,0);

%Halbkurve 5 zu -3
\fill[orange!40] (5,0) -- (-3,0) arc (180:360:4cm) -- cycle;
\draw[blue] (-3,0) arc (180:360:4cm);

%Halbkurve -3 zu -2
\fill[orange!40] (-3,0) -- (-2,0) arc (0:180:0.5cm)-- cycle;
\draw[blue] (-2,0) arc (0:180:0.5cm);

%Halbkurve -2 zu -1
\fill[green!20] (-2,0) -- (-1,0) arc (360:180:0.5cm)-- cycle;
\draw[blue] (-1,0) arc (360:180:0.5cm);

%Halbkurve -1 zu 0
\fill[orange!40] (-1,0) -- (0,0) to [out=90,in=0] (-0.5,1) to [out=180,in=90] (-1,0);
\draw[blue] (0,0) to [out=90,in=0] (-0.5,1) to [out=180,in=90] (-1,0);

%Halbkurve 0 zu 3
\fill[green!20] (0,0) -- (3,0) arc (360:180:1.5cm)-- cycle;
\draw[blue] (3,0) arc (360:180:1.5cm);

%Halbkurve 3 zu 2
\fill[green!20] (3,0) -- (2,0) arc (180:0:0.5cm)-- cycle;
\draw[blue] (2,0) arc (180:0:0.5cm);

%Halbkurve 2 zu 1
\fill[orange!40] (2,0) -- (1,0) arc (180:360:0.5cm)-- cycle;
\draw[blue] (1,0) arc (180:360:0.5cm);

\fill[blue] (-5,0) circle [radius=4pt];
\fill[blue] (-4,0) circle [radius=4pt];
\fill[blue] (-3,0) circle [radius=4pt];
\fill[blue] (-2,0) circle [radius=4pt];
\fill[blue] (-1,0) circle [radius=4pt];
\fill[blue] (-0,0) circle [radius=4pt];
\fill[blue] (1,0) circle [radius=4pt];
\fill[blue] (2,0) circle [radius=4pt];
\draw[blue] (-5,0) node[below=4pt,left=-3pt] {$s_1$};
\draw[blue] (-4,0) node[above=4pt,left=-3pt] {$s_2$};
\draw[blue] (-3,0) node[below=4pt,left=-3pt] {$s_3$};
\draw[blue] (-2,0) node[below=4pt,left=-3pt] {$s_4$};
\draw[blue] (-1,0) node[above=4pt,left=-3pt] {$s_5$};
\draw[blue] (0,0) node[below=4pt,left=-3pt] {$s_6$};
\draw[blue] (1,0) node[below=4pt,left=-3pt] {$s_7$};
\draw[blue] (2,0) node[below=4pt,right=-3pt] {$s_8$};

\draw (-6,0) -- (6,0);
\draw[thin] (-6,-6) -- (-6,6);
\draw[thin] (6,-6) -- (6,6);

\draw (-4,2) node[above=0pt, blue] {$L$};
\draw (1.5,0) node[above=-2pt] {$L_0$};

\end{tikzpicture}
\end{center}
% \subcaption{Chart.}
  \end{minipage}
    \captionsetup{format=hang}
\caption{On the right, $C_1$ is coloured in green and $C_2$ in orange.
    On the left, $R_1$ is in green and $R_2$ in orange for the upper 
    smooth lune from $s_1$ to $s_8$.}
    \label{fig:leaves}
\end{figure}
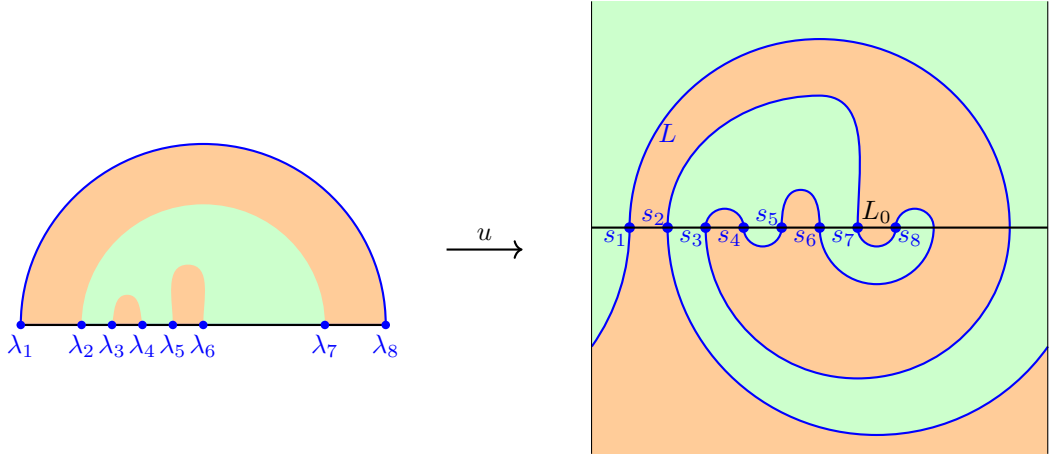

Then $R_1 \cap R_2 = \emptyset$.
Let $0 = \lambda_1 \leq \dots \leq \lambda_l =1$ be the points on
$\mathbb{D} \cap \R$ with $u(\lambda_i) = s_i$.
Then $[\lambda_i, \lambda_{i+1}] \subset \overline{R_1}$ for odd $i$
and $[\lambda_i, \lambda_{i+1}] \subset \overline{R_2}$ for even $i$.
Moreover,
$[\lambda_1,\lambda_2]$ is in the same connected component of $\overline{R_1}$ as
$[\lambda_{l-1}, \lambda_l]$.
We set $j_1 = l-1$.
Note that $j_1 \geq 3$.
We claim that there exists $4 \leq j_2 < j_1$,
such that $[\lambda_2, \lambda_3]$ and $[\lambda_{j_2}, \lambda_{j_2+1}]$
are in the same connected component of $\overline{R_2}$.
If not, consider the connected component $R$ of $[\lambda_2,\lambda_3]$
in $\overline{R_2}$. $R \cap \partial \mathbb{D} = [\lambda_2, \lambda_3]$
because $\mathrm{Int}(R)\cap R_1 = \emptyset$.
Therefore, $u$ restricts to a diffeomorphism from $R$ to $u(R)$. $\partial R \backslash [\lambda_2,\lambda_3] \subset L$ and hence $u(R)$ is a region in $\Sigma \backslash (L_0 \cup L)$ that is bounded by $[s_2,s_3]$ and part of $L$. Therefore, $u(R)$ is exactly the region corresponding to the subtree $T_{v(e_j)}$.
But then $u(R)$ contains a leaf, which contradicts our assumption. We therefore find a $j_2$ as claimed.

Proceeding like this, we obtain an infinite sequence $\{j_k\}_{k \in \N}$
of natural numbers
with $$1 < j_k < j_{k-1} < l.$$ This is impossible.
We conclude that $u$ does contain a leaf.
\end{proof}

%%%%%%%%%%%%%%%%%%%%% References %%%%%%%%%%%%%%%%%%%%%
\ \\
\bibliographystyle{aomalpha}
\bibliography{main}

\end{document}